\documentclass[10pt]{alggeom}
\usepackage{amsfonts}
\usepackage{mathrsfs}
\usepackage{amscd}
\usepackage{amsmath}
\usepackage{amssymb}
\usepackage{amsbsy} 
\usepackage{latexsym}
\usepackage{lscape}
\usepackage{comment}
\usepackage{amscd}
\usepackage{wasysym}  
\usepackage{tikz}
\usetikzlibrary{matrix,arrows}
\usepackage{tikz-cd}
\usepackage{appendix}
\usepackage[nice]{nicefrac}
\usepackage{dsfont, stmaryrd,enumerate}
\usepackage{xparse}

\usepackage[pagebackref,hyperindex,linktocpage=true]{hyperref}
\hypersetup{
    colorlinks,
    linkcolor={red!50!black},
    citecolor={blue!50!black},
    urlcolor={blue!80!black}
}

\usepackage{color}

\usetikzlibrary{decorations.markings}

\makeatletter
\tikzcdset{
  open/.code     = {\tikzcdset{hook, circled};},
  closed/.code   = {\tikzcdset{hook, slashed};},
  open'/.code    = {\tikzcdset{hook', circled};},
  closed'/.code  = {\tikzcdset{hook', slashed};},
  circled/.code  = {\tikzcdset{markwith = {\draw (0,0) circle (.375ex);}};},
  slashed/.code  = {\tikzcdset{markwith = {\draw[-] (-.4ex,-.4ex) -- (.4ex,.4ex);}};},
  markwith/.code ={
    \pgfutil@ifundefined%
    {tikz@library@decorations.markings@loaded}%
    {\pgfutil@packageerror{tikz-cd}{You need to say %
      \string\usetikzlibrary{decorations.markings} to use arrows with markings}{}}{}%
    \pgfkeysalso{/tikz/postaction = {
      /tikz/decorate,
      /tikz/decoration={markings, mark = at position 0.5 with {#1}}}
    }
  },
}
\makeatother

\renewcommand{\geq}{\geqslant}
\renewcommand{\leq}{\leqslant}
\renewcommand{\ge}{\geqslant}
\renewcommand{\le}{\leqslant}

\newtheorem{thm}{Theorem}[section]

\newtheorem{propo}[thm]{Proposition}

\newtheorem{propodef}[thm]{Definition and Proposition}
\newtheorem{lem}[thm]{Lemma}
\newtheorem{sublem}[thm]{Sublemma}
\newtheorem{lem-def}[thm]{Lemma-Definition}
\newtheorem{cor}[thm]{Corollary}
\newtheorem{conject}[thm]{Conjecture}
\newtheorem{propert}[thm]{Properties}
\newtheorem{observ}[thm]{Observation}

\newtheorem{fac}[thm]{Fact}

\newtheorem{notat}[thm]{Notation}
\theoremstyle{definition}
\newtheorem*{ack}{Acknowledgement}

\newtheorem{ex}[thm]{Example}

\newtheorem{rmk}[thm]{Remark}
\newtheorem{dfn}[thm]{Definition}
\newtheorem{quest}[thm]{Question}

\newtheorem*{abs}{Abstract}
\numberwithin{equation}{section}

\newcommand{\nc}{\newcommand}

\nc{\abst}{\begin{abs}} \nc{\xabst}{\end{abs}}
\nc{\theo}{\begin{thm}} \nc{\xtheo}{\end{thm}}
\nc{\prop}{\begin{propo}} \nc{\xprop}{\end{propo}}

\nc{\nota}{\begin{notat}} \nc{\xnota}{\end{notat}}
\nc{\depr}{\begin{propodef}} \nc{\xdepr}{\end{propodef}}
\nc{\lemm}{\begin{lem}} \nc{\xlemm}{\end{lem}}
\nc{\sublemm}{\begin{sublem}} \nc{\xsublemm}{\end{sublem}}
\nc{\lemmdefi}{\begin{lem-def}} \nc{\xlemmdefi}{\end{lem-def}}
\nc{\coro}{\begin{cor}} \nc{\xcoro}{\end{cor}}
\nc{\conj}{\begin{conject}} \nc{\xconj}{\end{conject}}
\nc{\proper}{\begin{propert}} \nc{\xproper}{\end{propert}}
\nc{\obse}{\begin{observ}} \nc{\xobse}{\end{observ}}
\nc{\ques}{\begin{quest}} \nc{\xques}{\end{quest}}

\nc{\fact}{\begin{fac}} \nc{\xfact}{\end{fac}}
\nc{\ackn}{\begin{ack}} \nc{\xackn}{\end{ack}}
\nc{\exam}{\begin{ex}} \nc{\xexam}{\end{ex}}
\nc{\rema}{\begin{rmk}} \nc{\xrema}{\end{rmk}}
\nc{\defi}{\begin{dfn}} \nc{\xdefi}{\end{dfn}}

\nc{\pf}{\begin{proof}} \nc{\xpf}{\end{proof}}

\nc{\on}{\operatorname}
\nc{\fraka}{{\mathfrak a}} \nc{\bba}{{\mathbf a}}
\nc{\frakb}{{\mathfrak b}}
\nc{\frakc}{{\mathfrak c}}
\nc{\frakd}{{\mathfrak d}}
\nc{\frake}{{\mathfrak e}}
\nc{\frakf}{{\mathfrak f}}
\nc{\frakg}{{\mathfrak g}}
\nc{\frakh}{{\mathfrak h}}
\nc{\fraki}{{\mathfrak i}}
\nc{\frakj}{{\mathfrak j}}
\nc{\frakk}{{\mathfrak k}}
\nc{\frakl}{{\mathfrak l}}
\nc{\frakm}{{\mathfrak m}}
\nc{\frakn}{{\mathfrak n}}
\nc{\frako}{{\mathfrak o}}
\nc{\frakp}{{\mathfrak p}}
\nc{\frakq}{{\mathfrak q}}
\nc{\frakr}{{\mathfrak r}}
\nc{\fraks}{{\mathfrak s}}
\nc{\frakt}{{\mathfrak t}}
\nc{\fraku}{{\mathfrak u}}
\nc{\frakv}{{\mathfrak v}}
\nc{\frakw}{{\mathfrak w}}
\nc{\frakx}{{\mathfrak x}}
\nc{\fraky}{{\mathfrak y}}
\nc{\frakz}{{\mathfrak z}}
\nc{\frakA}{{\mathfrak A}}
\nc{\frakB}{{\mathfrak B}}
\nc{\frakC}{{\mathfrak C}}
\nc{\frakD}{{\mathfrak D}}
\nc{\frakE}{{\mathfrak E}}
\nc{\frakF}{{\mathfrak F}}
\nc{\frakG}{{\mathfrak G}}
\nc{\frakH}{{\mathfrak H}}
\nc{\frakI}{{\mathfrak I}}
\nc{\frakJ}{{\mathfrak J}}
\nc{\frakK}{{\mathfrak K}}
\nc{\frakL}{{\mathfrak L}}
\nc{\frakM}{{\mathfrak M}}
\nc{\frakN}{{\mathfrak N}}
\nc{\frakO}{{\mathfrak O}}
\nc{\frakP}{{\mathfrak P}}
\nc{\frakQ}{{\mathfrak Q}}
\nc{\frakR}{{\mathfrak R}}
\nc{\frakS}{{\mathfrak S}}
\nc{\frakT}{{\mathfrak T}}
\nc{\frakU}{{\mathfrak U}}
\nc{\frakV}{{\mathfrak V}}
\nc{\frakW}{{\mathfrak W}}
\nc{\frakX}{{\mathfrak X}}
\nc{\frakY}{{\mathfrak Y}}
\nc{\frakZ}{{\mathfrak Z}}
\nc{\bbA}{{\mathbb A}}
\nc{\bbB}{{\mathbb B}}
\nc{\bbC}{{\mathbb C}}
\nc{\bbD}{{\mathbb D}}
\nc{\bbE}{{\mathbb E}}
\nc{\bbF}{{\mathbb F}} \nc{\bbf}{{\mathbf f}}
\nc{\bbG}{{\mathbb G}}
\nc{\bbH}{{\mathbb H}}
\nc{\bbI}{{\mathbb I}}
\nc{\bbJ}{{\mathbb J}}
\nc{\bbK}{{\mathbb K}}
\nc{\bbL}{{\mathbb L}}
\nc{\bbM}{{\mathbb M}}
\nc{\bbN}{{\mathbb N}}
\nc{\bbO}{{\mathbb O}}
\nc{\bbP}{{\mathbb P}}
\nc{\bbQ}{{\mathbb Q}}
\nc{\bbR}{{\mathbb R}}
\nc{\bbS}{{\mathbb S}}
\nc{\bbT}{{\mathbb T}}
\nc{\bbU}{{\mathbb U}}
\nc{\bbV}{{\mathbb V}}
\nc{\bbW}{{\mathbb W}}
\nc{\bbX}{{\mathbb X}}
\nc{\bbY}{{\mathbb Y}}
\nc{\bbZ}{{\mathbb Z}}
\nc{\calA}{{\mathcal A}}
\nc{\calB}{{\mathcal B}}
\nc{\calC}{{\mathcal C}}
\nc{\calD}{{\mathcal D}}
\nc{\calE}{{\mathcal E}}
\nc{\calF}{{\mathcal F}}
\nc{\calG}{{\mathcal G}}
\nc{\calH}{{\mathcal H}}
\nc{\calI}{{\mathcal I}}
\nc{\calJ}{{\mathcal J}}
\nc{\calK}{{\mathcal K}}
\nc{\calL}{{\mathcal L}}
\nc{\calM}{{\mathcal M}}
\nc{\calN}{{\mathcal N}}
\nc{\calO}{{\mathcal O}}
\nc{\calP}{{\mathcal P}}
\nc{\calQ}{{\mathcal Q}}
\nc{\calR}{{\mathcal R}}
\nc{\calS}{{\mathcal S}}
\nc{\calT}{{\mathcal T}}
\nc{\calU}{{\mathcal U}}
\nc{\calV}{{\mathcal V}}
\nc{\calW}{{\mathcal W}}
\nc{\calX}{{\mathcal X}}
\nc{\calY}{{\mathcal Y}}
\nc{\calZ}{{\mathcal Z}}

\nc{\Spac}{{\mathrm{Spaces}}}
\nc{\scrA}{{\mathscr A}}
\nc{\scrE}{{\mathscr E}}
\nc{\scrR}{{\mathscr R}}

\nc{\Bmu}{\mbox{$\raisebox{-0.59ex}{$l$}\hspace{-0.18em}\mu\hspace{-0.88em}\raisebox{-0.98ex}{\scalebox{2}{$\color{white}.$}}\hspace{-0.416em}\raisebox{+0.88ex}{$\color{white}.$}\hspace{0.46em}$}{}}

\nc{\bnu}{{\bar{ \nu}}}

\nc{\olO}{\bar{\calO}}

\nc{\al}{{\alpha}} 
\nc{\be}{{\beta}}
\nc{\ga}{{\gamma}} \nc{\Ga}{{\Gamma}}
 \nc{\hGa}{\hat{\Gamma}}
\nc{\ve}{{\varepsilon}} 
\nc{\la}{{\lambda}} \nc{\La}{{\Lambda}}
\nc{\om}{\omega} \nc{\Om}{\Omega} 
\nc{\sig}{{\sigma}} \nc{\Sig}{{\Sigma}}

\nc{\tnb}{\psi_{\rm tame}}
\nc{\oM}{\overline{{M}}}
\nc{\op}{{\on{op}}}
\nc{\ad}{{\on{ad}}}
\nc{\alg}{{\on{alg}}}
\nc{\Ad}{{\on{Ad}}}
\nc{\Adm}{{\on{Adm}}} \nc{\aff}{{\on{aff}}}
\nc{\Aut}{{\on{Aut}}}
\nc{\Bun}{{\on{Bun}}}
\nc{\cha}{{\on{char}}}
\nc{\der}{{\on{der}}}
\nc{\Der}{{\on{Der}}}
\nc{\diag}{{\on{diag}}}
\nc{\End}{{\on{End}}}
\nc{\Fl}{{\calF\!\ell}}
\nc{\Tr}{{\on{Transp}}}
\nc{\TR}{{\calT\!\calR}}
\nc{\Gal}{{\on{Gal}}}
\nc{\Gr}{{\on{Gr}}}
\nc{\rH}{{\on{H}}}
\nc{\Hom}{{\on{Hom}}}
\nc{\IC}{{\on{IC}}}
\nc{\id}{{\on{id}}}
\nc{\Id}{{\on{Id}}}
\nc{\ind}{{\on{ind}}}
\nc{\Ind}{{\on{Ind}}}
\nc{\Lie}{{\on{Lie}}}
\nc{\Pic}{{\on{Pic}}}
\nc{\pr}{{\on{pr}}}
\nc{\Res}{{\on{Res}}}
\nc{\res}{{\on{res}}} \nc{\Sat}{{\on{Sat}}}
\nc{\s}{{\on{sc}}}
\nc{\drv}{{\on{der}}}
\nc{\sgn}{{\on{sgn}}}
\nc{\Spec}{{\on{Spec}}}\nc{\Spf}{\on{Spf}} 
\nc{\Sph}{\on{Sph}}
\nc{\St}{{\on{St}}}
\nc{\tr}{{\on{tr}}}
\nc{\Mod}{{\mathrm{-Mod}}}
\nc{\Hilb}{{\on{Hilb}}} 
\nc{\Ext}{{\on{Ext}}} 
\nc{\vs}{{\on{Vec}}}
\nc{\ev}{{\on{ev}}}
\nc{\nO}{{\breve{\calO}}}
\nc{\tS}{{\tilde{S}}}
\nc{\spe}{{\on{sp}}}
\nc{\loc}{{\on{loc}}}
\nc{\Sym}{{\on{Sym}}}
\nc{\Cone}{{\on{C}}}
\nc{\syn}{{\on{syn}}}
\nc{\reg}{{\on{reg}}}
\nc{\colim}{{\on{colim}}}
\nc{\Norm}{{\on{N}}}

\nc{\nscrR}{{\mathscr{R}^{\on{nr}}}}

\nc{\GL}{{\on{GL}}}
\nc{\U}{{\on{U}}}
\nc{\Gl}{\on{Gl}} 
\nc{\GSp}{{\on{GSp}}}
\nc{\gl}{{\frakg\frakl}}
\nc{\SL}{{\on{SL}}} 
\nc{\SU}{{\on{SU}}} 
\nc{\SO}{{\on{SO}}}
\nc{\PGL}{{\on{PGL}}}

\nc{\Conv}{{\on{Conv}}}
\nc{\Rep}{{\on{Rep}}}
\nc{\Dom}{{\on{Dom}}}
\nc{\red}{{\on{red}}}
\nc{\act}{{\on{act}}}
\nc{\nr}{{\on{nr}}}
\nc{\ctf}{{\on{ctf}}}

\nc{\str}{{\on{-}}} 
\nc{\os}{{\bar{s}}}
\nc{\oeta}{{\bar{\eta}}}

\nc{\hookto}{\hookrightarrow}
\nc{\longto}{\longrightarrow}
\nc{\leftto}{\leftarrow}
\nc{\onto}{\twoheadrightarrow}
\nc{\lonto}{\twoheadleftarrow}
\newcommand*\isomto{%
  \renewcommand{\arraystretch}{0.1}
  \begin{array}[b]{c} {}_{\sim} \\ \longrightarrow \end{array}%
}

\nc{\uG}{{\underline{G}}}
\nc{\uA}{{\underline{A}}}
\nc{\uS}{{\underline{S}}}
\nc{\uT}{{\underline{T}}}
\nc{\uM}{{\underline{M}}}
\nc{\uP}{{\underline{P}}}
\nc{\uB}{{\underline{B}}}
\nc{\uN}{{\underline{N}}}

\nc{\ucG}{{\underline{\calG}}}
\nc{\ucA}{{\underline{\calA}}}
\nc{\ucS}{{\underline{\calS}}}
\nc{\ucT}{{\underline{\calT}}}
\nc{\ucalM}{{\underline{\calM}}}
\nc{\ucP}{{\underline{\calP}}}
\nc{\ucalN}{{\underline{\calN}}}

\nc{\bF}{{\breve{F}}}

\nc{\oFl}{{\overline{\Fl}}} 
\nc{\bU}{{\overline{U}}}
\nc{\tGr}{{\tilde{\Gr}}}
\nc{\cGr}{\calG\! r}
\nc{\oGr}{\overline{\on{Gr}}} 
\nc{\ocGr}{\overline{\calG\! r}}
\nc{\co}{{\colon}}
\nc{\sch}[1]{(Sch/{#1})}
\nc{\HypLoc}[1]{HypLoc({#1})}

\nc{\ohtimes}{\stackrel{!}{\otimes}}
\nc{\boxtilde}{\widetilde{\boxtimes}}
\nc{\vstar}{{\varhexstar}}

\nc{\Div}{\on{Div}}
\nc{\Sht}{\on{Sht}}
\nc{\Frob}{\on{Frob}}

\nc{\x}{\times}
\nc{\bsl}{\backslash}
\nc{\algQl}{{\bar{\bbQ}_\ell}}
\nc{\sF}{{\bar{F}}}
\nc{\nF}{{\breve{F}}}
\nc{\nW}{{W^{\on{nr}}}}
\nc{\sk}{{\bar{k}}}
\nc{\cont}{\on{c}}
\nc{\Supp}{\on{Supp}}
\nc{\blt}{\bullet}  
\nc{\dom}{\on{dom}}
\nc{\scon}{{\on{sc}}} 
\nc{\Affine}{\on{Aff}} 
\nc{\nscrA}{\mathscr{A}^{\on{nr}}} 
\nc{\nfraka}{{\bbf^{\on{nr}}}}
\nc{\ran}{{\rangle}}
\nc{\lan}{{\langle}}
\nc{\bk}{{\bar{k}}}
\nc{\tF}{{\tilde{F}}}
\nc{\sS}{{\bar{S}}}
\nc{\LG}{{^\text{L}\hspace{-0.04cm}G}}
\nc{\LL}{{^\text{L}\hspace{-0.07cm}L}}
\nc{\et}{{\text{\rm \'et}}}
\nc{\inv}{{\on{inv}}}
\nc{\Hecke}{{\on{Hecke}}}
\nc{\Isom}{{\on{Isom}}}
\nc{\oSht}{{\overline{\on{Sht}}}}
\nc{\umu}{{\underline \mu}}
\nc{\AIJ}{{\calO_X[{\scriptstyle{\calI\over \calJ}}]}}
\nc{\Proj}{{\on{Proj}}}
\nc{\Bl}{{\on{Bl}}}

\nc{\Pos}{{\on{Pos}}}
\nc{\Sets}{{\on{Sets}}}
\nc{\AffSch}{{\on{AffSch}}}
\nc{\Groups}{{\on{Groups}}}
\nc{\Gpds}{{\on{Groupoids}}}
\nc{\Sch}{{\on{Sch}}}
\nc{\fl}{{\on{flat}}}

\nc{\pot}[1]{ [\hspace{-0,5mm}[ {#1} ]\hspace{-0,5mm}] }
\nc{\rpot}[1]{ (\hspace{-0,7mm}( {#1} )\hspace{-0,7mm}) }

\nc{\defined}{\hspace{0.1cm}\stackrel{\text{\tiny \rm def}}{=}\hspace{0.1cm}}
\usepackage{bbm}
\usepackage{array,multirow,makecell}
\setcellgapes{1pt}
\makegapedcells
\newcolumntype{R}[1]{>{\raggedleft\arraybackslash }b{#1}}
\newcolumntype{L}[1]{>{\raggedright\arraybackslash }b{#1}}
\newcolumntype{C}[1]{>{\centering\arraybackslash }b{#1}}

\begin{document}

\title{A survey on algebraic dilatations}

\shortauthors{{ Adrien Dubouloz, Arnaud Mayeux and João Pedro dos Santos }}


\maketitle
\[\text{{\large {Adrien Dubouloz, Arnaud Mayeux and João Pedro dos Santos }}}\]

\[\text{\today} \]

$~~$

\abst
In this text, we wish to provide the reader with a short guide to recent works on  the theory of  dilatations in  Commutative Algebra and Algebraic Geometry. These works fall naturally into two categories: one emphasises {\it foundational and theoretical} aspects and the other {\it applications} to existing theories.
\xabst

\tableofcontents

\bigskip 
\bigskip 
\bigskip

{\small\ackn
 {{This project has received funding from the European Research Council (ERC) under the European Union’s Horizon 2020 research and innovation programme (A.\,M. grant agreement No 101002592). A.\,M. is supported by ISF grant 1577/23.}}
\xackn}

\newpage

\section*{Introduction}

\begin{center}
\textbf{\textit{{What is the concept of algebraic dilatations about ? }}}
\end{center}
Dilatation of rings is a basic construction of commutative algebra, like localization or tensor product. It can be globalized so that it also makes sense on schemes or algebraic spaces. In fact dilatations generalize localizations.

Let $A$ be a ring and let $S$ be a multiplicative subset of $A$. Recall that the localization $S^{-1}A$ is an $A$-algebra such that for any $A$-algebra $ A \to B $ such that the image of $s$ is invertible for any $s \in S$, then $A \to B$ factors through $A \to S^{-1}A$. Intuitively, $S^{-1}A$ is the $A$-algebra obtained from $A$ adding all fractions $\frac{a}{s}$ with $a \in A $ and $s \in S$.
Formally, $S^{-1} A$ is made of classes of fractions $\frac{a}{s}$ where $a \in A$ and $s \in S $ (two representatives $\frac{a}{s}$ and $\frac{b}{t}$ are identified if $atr=bsr$ for some $r \in S$), addition and multiplication are given by usual formulas. Now let us give for any element $s \in S$ an ideal $M_s$ of $A$ containing $s$. The dilatation of $A$ relatively to the data $ \{M_s \}_{s \in S}$ is an $A$-algebra $A'$ obtained intuitively by adding to $A$ only the fractions $\frac{m}{s}$ with $s \in S $ and $m \in M_s$. The dilatation $A'$ satisfies that for any $s \in S$, we have $s A'= M_s A' $ (intuitively any $m \in M_s$ belongs to $s A'$, i.e. becomes a multiple of $s$, so that we have an element $\frac{m}{s} $ such that $m = s  \frac{m}{s} $). As a consequence of the construction, the elements $s \in S$ become non-zero-divisor in $A'$ so that $\frac{m}{s}$ is well-defined (i.e. unique). It turns out that it is convenient, with dilatations of schemes in mind, to make a bit more flexible the above framework, namely to remove the conditions that $S$ is multiplicative and that $s \in M_s$, so we use the following definition. 

\begin{flushleft}
\textsc{Definition.} Let $A$ be a ring. Let $I$ be an index set. A multi-center in $A$ indexed by $I$ is a set of pairs $\{[M_i , a_i ]\}_{i \in I}$ where for each $i$, $M_i$ is an ideal of $A$ and $a_i$ is an element of $A$.
\end{flushleft}

To each multi-center $\{[M_i , a_i ]\}_{i \in I}$, one has the dilatation $A [\{\frac{M_i}{a_i}\}_{i \in I}]$ which is an $A$-algebra.  We will define and study in details dilatations of rings in Section \ref{09.05.2023--1}, in particular we will state formally the universal property they enjoy. We will also see that $A [\{\frac{M_i}{a_i}\}_{i \in I}]$ is generated, as $A$-algebra, by $\{\frac{M_i}{a_i}\}_{i \in I}$. We will also see that if $M_i = A$ for all $i$, then $A [\{\frac{M_i}{a_i}\}_{i \in I}]= S^{-1} A$ where $S$ is the multiplicative subset generated by $\{a _i \} _{i \in I }$. Reciprocally, we will see that any sub-$A$-algebra of a localization $S^{-1} A$ for a certain $S$ is isomorphic to a dilatation of $A$. 

Dilatations of schemes and algebraic spaces are obtained from dilatations of rings via glueing. We introduce the following definition.

\begin{flushleft}

\textsc{Definition.} Let $X$ be a scheme. Let $I$ be an index set. A multi-center in $X$ indexed by $I$ is a set of pairs $\{[Y_i, D_i]\}_{i \in I } $ such that $Y_i$ and $D_i $ are closed subschemes for each $i$ and such that locally, all $D_i$ are principal for $i \in I$.
\end{flushleft}

Associated to each multi-center, one has the dilatation $\Bl \big\{{}^{D_i}_{Y_i} \big\}_{i \in I } X $ which is a scheme endowed with a canonical affine morphism $f: \Bl \big\{{}^{D_i}_{Y_i} \big\}_{i \in I } X \to X$. It satisfies, in a universal way, that $f^{-1} (D_i)$ is a Cartier divisor (i.e. is locally defined by a non-zero-divisor) and that $f^{-1} (D_i ) = f^{-1} (Y_i \cap D_i)$ for all $i \in I$. If $\#I=1$, we use the terminology mono-centered dilatation.
 We will study several facets of this construction and show that it enjoys many wonderfull properties in Sections \ref{333} and \ref{neron}.

\begin{center}
\textbf{\textit{{Where does this construction come from ? }}}
\end{center}

As we saw in the previous section, dilatations are a basic construction which can be easily encountered in specific situations. As a consequence it was used for a very long time.
As the reader may well know, the theory of dilatations has deep and distinguished roots, even though not formulated in the language which we use. Right from the start, we warn the reader that we do not mean to, and probably could not, present a comprehensive historical account.
 As soon as Cremona and Bertini started using quadratic transformations (or blowups) in the framework of algebraic geometry over fields, ``substitutions'' of the form $x'=x$ and $y'=y/x$ started being made by algebraic geometers, see for example equation (8) in \cite[Section 11]{No1884} and Noether's acknowledgement, at the start of \cite[Section 12]{No1884}, that these  manipulations come from Cremona's point of view. Examples of dilatations appear frequently in some works of Zariski and Abhyankar, cf. \cite[Definition, p. 86]{abhyankar-zariski55} and \cite[p499 proof of Th.4, case (b)]{Za43}. Other forerunner examples of dilatations play a central role in several independent and unrelated works later,  cf. \cite{Da67},  \cite[Section 25]{Ner64} and \cite[Section 4]{artin69}. 
As far as we know, the terminology of dilatations emerged in \cite[§3.2]{BLR90}, where a section is devoted to study dilatations of schemes over discrete valuation rings systematically. The treatment of dilatations of \cite{BLR90} was used in \cite{GY00}, \cite{CY01}, \cite{Ch18} and many other references to construct models. Still in the context of schemes over a discrete valuation ring, we draw the reader's attention to \cite{Ana73},  \cite{WW80} and \cite{PY06} which contain results on dilatations. 
Certain dilatations also appear as the central geometric tool in the study of the ramification theory of $\ell$-adic sheaves in positive characteristic \cite{AbSa07,AbSa09,Sai09}. 
The paper \cite{KZ99} studies dilatations (under the name affine modifications) systematically in the framework of algebraic geometry over fields. Over two-dimensional base schemes dilatations also appear in \cite[p. 175]{PZ13}.
Set aside localizations, mono-centered dilatations have been the main focus of mathematicians in the past. However, in the context of group schemes over discrete valuation rings, examples of multi-centered dilatations of rings and schemes that are not localizations or mono-centered dilatations appeared and were used in \cite[Exp. VIB Ex. 13.3]{DG70}, \cite{PY06} and \cite{duong-hai-dos_santos18}. In recent times, the authors of \cite{Du05}, \cite{MRR20} and \cite{Ma23d} have set out to accommodate all these constructions in a larger and unified frame, namely for arbitrary schemes and algebraic spaces and arbitrary multi-center. The paper \cite{MRR20} introduces dilatations of arbitrary schemes in the mono-centered case and provides a systematic treatment of mono-centered dilatations of general schemes. An equivalent definition of mono-centered dilatations of general schemes, under the name affine modifications, was introduced earlier in \cite[Définition 2.9]{Du05} under a few assumptions. The paper \cite{Ma23d} introduces and studies dilatations of arbitrary rings, schemes and algebraic spaces for arbitrary multi-centers. Allowing multi-centers also leads naturally to the formulation of combinatorial isomorphisms on dilatations and gives birth to refined universal properties.
Nevertheless, the mono-centered case remains a fundamental case that is frequently the 'atom' for some aspects of the theory. 

The first part (Sections \ref{09.05.2023--1}-\ref{neron}) of this survey is devoted to theoretical and formal results on dilatations of rings, schemes and group schemes following \cite{Du05}, \cite{MRR20} and \cite{Ma23d}. Sections \ref{Tannak}-\ref{Topol} of the second part will deal with several concrete situations where specific kinds of dilatations play a role, annd will provide complementary inceptions on this construction.
To finish, beyond rings and schemes, the concept of dilatations makes sense for other structures and geometric settings. Let us indicate some constructions already available.
Some dilatation constructions in the framework of complex analytic spaces were introduced in \cite{Ka94}, these are used and discussed in Section \ref{Topol}. Dilatations also make sense for general algebraic spaces \cite{Ma23d}. Similarly, for many other structures than rings, dilatations also make sense (e.g. categories, monoids and semirings) as noticed in \cite{Ma23c}. It is possible that dilatations in other settings will be explored and find a significant role since, at the end, these are a basic mathematical concept.

\begin{center}
\textbf{\textit{Terminology}}
\end{center}
Recall that dilatations have distinguished  roots. As a consequence, several other terminologies are used to call certains dilatations in literature.  For examples the constructions named \textit{affine blowups}, \textit{affine modifications}, \textit{automatic blowups}, \textit{formal blowups}, \textit{Kaliman modifications}, \textit{localizations} and \textit{Néron blowups} are examples of (eventually multi-centered) dilatations.

\begin{center}
\textbf{{\textit{Some simple examples}}}
\end{center} 

We provided an intuition on dilatations of rings before. Let us now provide some simple examples of dilatations of schemes. If $S$ is a scheme, then we denote by $e_S$ the trivial group scheme over $S$, as scheme it is isomorphic to $S$. If $G $ is a separated group scheme over $S$, then we denote by $e_G$ the trivial closed group scheme. Note that $e_G $ is isomorphic to $e_S$ as group schemes over $S$.
\begin{enumerate}
 \item  We consider, once given a prime number $p$, the  multiplicative  group scheme \[G= \mathbb G_{m,\bbZ _p}\] over the ring $\mathbb Z_p$; its Hopf algebra is $A=\mathbb Z_p[x,x^{-1}]$ while the morphism $\Delta:A\to A\otimes _{\bbZ _p} A$  induced from multiplication $G\times _ {\bbZ_p} G \to G$ is defined by
  $\Delta(x)= x\otimes x$. Now, consider the couple $ e_{G} \text{ and }  G \times _{\bbZ _p} \bbZ _p / \frakp^r \text{ of closed subschemes of } G $ for any $r >0$ ($\frakp^r$ denotes $p^r\bbZ _p$). These are cut out, respectively, by the ideals $I=(x-1)$ and $(p^r)$ of $A$.
\begin{enumerate} \item For any $r >0$, the dilatation $A'$ of $A$ centered at $[e_G ,  G \times _{\bbZ _p} \bbZ _p / \frakp^r]$, or at $[I,(p^r)]$,  is the sub-$A$-algebra of $A[1/p]$ generated by all the elements $p^{-r}f$, where $f\in I$.  The dilatation $G':=\mathrm {Spec}\, A'$ is a group scheme of finite type over $\bbZ_p$. The base change $G ' \times _{\bbZ _p} \bbZ _p / \frakp ^r$ is isomorphic to the {\it additive} group $\mathbb G_a$ over $ \bbZ_p /\frakp ^r$, while $G' \times _{\bbZ _p} \bbQ _p$ is the {\it multiplicative} group $\mathbb G_m$ over $\mathbb Q_p$. Furthermore, on the level of points, $G'(\bbZ_p)=1+\frakp^r$ is a {\it congruence subgroup}.
\item The dilatation $A'$ of $A$ centered at $\{ [e_G ,  G \times _{\bbZ _p} \bbZ _p / \frakp^r ] \} _{\{r > 0 \}}$,
 or at $\{[I,(p^r)]\}_{\{r >0\}}$,  is the sub-$A$-algebra of $A[1/p]$ generated by all the elements $p^{-r}f$, where $f\in I$ and $r > 0 $. 
 The dilatation $G':=\mathrm {Spec}\, A'$ is a group scheme over $\bbZ_p$ which is not of finite type. The base change $G ' \times _{\bbZ _p} \bbZ _p / \frakp ^r$ is isomorphic to the {\it trivial} group scheme over $ \bbZ_p /\frakp ^r$, while $G' \times _{\bbZ _p} \bbQ _p$ is the {\it multiplicative} group $\mathbb G_m$ over $\mathbb Q_p$. Furthermore, on the level of points, $G'(\bbZ_p)=\{1\}$.
 \end{enumerate}
 \item Let $G $ be $GL_3$ over $\bbZ _p $ and let $H $ be  $GL_2 \times e _{\bbZ _p} $ diagonally inside $G$ and let $e_G \cong (e_{\bbZ _p})^3$ be the trivial closed subgroup of $G$. For any $r>0$, let $G \times _{\bbZ _p}\bbZ _p/\frakp^r$ be the base change of $G$ to $\bbZ /p^r \bbZ$.
  The dilatation $G' =\Bl \big\{{}^{G \times _{\bbZ _p}  \bbZ _p /\frakp^5 }_{~~H} , {}^{G \otimes \bbZ _p /\frakp ^2 }_{~e_G} \big\} G$ 
  is a group scheme over $\bbZ _p$. On the level of points, we have 
  
 \begin{align*}  GL_3 (\bbZ _p)  \supset G' ( \bbZ _p ) &= \begin{pmatrix} 1+ \frakp^2 & \frakp ^2 & \frakp ^5 \\ \frakp^2 & 1+ \frakp^2 & \frakp ^5 \\ \frakp^5 & \frakp^5 & 1+ \frakp^5 \end{pmatrix} \\ &= \left\{\begin{pmatrix} 1+ a & b & e\\ c & 1+ d & f \\ g & h & 1+ k \end{pmatrix} \middle| a, b , c, d \in \frakp ^{2} ~~ e,f,g,h,k \in \frakp ^5  \right\} . \end{align*}

\item Let $X = \bbA^1 = \Spec (\bbZ [T])$ be the affine line over $\bbZ$, let $0 \subset \bbA^1 $ be the closed subscheme defined by the ideal $(T)$ and let $\emptyset \subset \bbA ^1$ be the closed subscheme defined by the ideal $\bbZ[T]$. Then the dilatation $\Bl _{\emptyset} ^{D} X$ identifies with the open subscheme $\bbG _m$ of $\bbA^1$, indeed $\bbZ [T] [\frac{\bbZ[T]}{T}] \cong \bbZ [T , T^{-1}]$. This is an example of localization.

\item Let $X = \bbA ^2 = \Spec ( \bbZ [A,B]) $ be the affine plane over $\bbZ$, let $D\subset \bbA ^2$ be the line defined by the ideal $(A)$ and let $0\subset \bbA^2$ be the origin defined by the ideal $(A,B)$. Then $\Bl  ^{D}_{0}  X$ identifies with $\Spec (\bbZ [A,B,C]/(AC-B))$. Indeed, one has an isomorphism (e.g. by Proposition \ref{regupoly})
\[\bbZ[A,B][\frac{(A,B)}{A}]\cong \bbZ[A,B][\frac{(B)}{A}]  \cong \bbZ[A,B,C]/(AC-B) .\] The morphism $\Bl _0 ^D X \to X $ is given by  
$\bbZ [A,B] \to \bbZ [A,B,C]/(AC-B) , A,B \mapsto A,B $. At the level of points $\big(\Bl _0 ^D X \big)(\bbZ )$ is made of pairs $(a,b) \in \bbZ ^2$ such that $b$ is a multiple of $a.$

\item More advanced examples of dilatations in contextual situations are available in the second part of this survey.
\end{enumerate}

\begin{center}
\textbf{\textit{{What is the aim of this survey ? }}}
\end{center}

Recall that we wish to provide the reader with a short guide to recent works on  the theory of  dilatations. We do not mean to present a comprehensive account. We mainly concentrate on the contributions that ourselves were responsible for  \cite{Du05, DF18, duong-hai-dos_santos18,  Ma19t, hai-dos_santos21, MRR20, ADO21, Ma23d} and those which were our starting points \cite{artin69, Ana73, And01, WW80, BLR90, KZ99, Yu15, PY06, PZ13, HKO16}. 

Part I is devoted to an exposition of general definitions and results around the concept of algebraic dilatations introduced and proved in \cite{MRR20} and \cite{Ma23d}. Section \ref{09.05.2023--1} discusses dilatations of commutative unital rings. Section \ref{333} summarizes general results on dilatations of schemes. Section \ref{neron} focuses on dilatations of group schemes.

 Part II provides an overview on some applications of dilatations in various mathematical contexts. 
In Section \ref{Tannak}, we explain some recent applications of dilatations to the theory of affine group schemes and their representation categories in the case where objects are defined over a discrete valuation ring $R$. These were developed mainly in order to advance the study of Tannakian categories defined over $R$ and  appearing in geometry, such as  the case of group schemes associated to $\mathcal D$-modules \cite{And01,dos_santos09,duong-hai18,duong-hai-dos_santos18,hai-dos_santos21} and principal bundles with finite structure groups \cite{hai-dos_santos23}. After a brief introduction to Tannakian categories over $R$ (Section \ref{Tannak1}), we go  on to explain how to filter these categories by smaller ones and produce in this way the ``Galois-Tannaka'' group schemes \ref{Tannak2}. We show why N\'eron blowups are a fundamental tool for studying these groups and explain what has been done so far in order to exhaust Galois-Tannaka groups by means of N\'eron blowups, both in the case of mono-centered and multi-centered N\'eron blowups (cf. Section \ref{Tannak2} and Section \ref{Tannak3}).  
  In Section \ref{RepIso}, \textit{congruent isomorphisms},  formulated and stated using the language of dilatations, are discussed in relation with \textit{the Moy-Prasad isomorphism}, \textit{Bruhat-Tits buildings} and \textit{representations of $p$-adic groups}. Section \ref{TorShtuk} shows that many level structures
on \textit{moduli stacks of $G$-bundles} are encoded in torsors under dilatations and that this can be used to obtain integral models of \textit{shtukas}. Section \ref{Topol} discusses dilatations in \textit{affine geometry} and related progress in \textit{$\bbA ^1$-homotopy theory}.

 All results stated in this paper are proved in indicated references. 
This survey is an expository text and does not contain any new mathematical result. What is perhaps new is to summarize some aspects of several independent works involving dilatations in a single text. We hope this could be a source of inspiration for future works.

\section*{Some conventions and notations}

\begin{enumerate}[(1)]
\item All rings are unital and commutative, unless otherwise mentioned. 
\item Let   $(M,+)$ be a monoid. A submonoid  $F$ is a face of $M$ if whenever $x+y\in F$, then both $x$ and $y$ belong to $F$. 
\item If $R$ is a discrete valuation ring with field of fractions $K$, then, for each $R$-scheme, we call $X\otimes_R K$ the generic fibre of $X$. 
\item If $A$ is a ring, then $A\text{-}\mathbf{mod}$ is the category of finitely presented $A$-modules. 
\item If $G$ is a group scheme over a noetherian ring $R$, or an abstract group,  then we denote by  $\mathrm {Rep}_R(G)$ the category of all $R$--modules of finite type affording a representation of $G$ as explained in \cite{jantzen03}.
\end{enumerate}

\part{Algebraic dilatations}

In this part, we introduce formally dilatations of rings and schemes. Locally, dilatations of schemes will be studied through dilatations of rings.

\section{Dilatations of rings}\label{09.05.2023--1} 
 We summarily  present basic results on dilatations of rings following the more general path given in \cite{Ma23d}. 
Let $A$ be a  ring. A {\it center in $A$} is a pair $[M,a]$ consisting of an ideal $M\subset A$ and an element $a\in A$. A {\it multi-center} is a family of center indexed by some set. Let $I$ be an index set and let  $\{[M_i , a_i] \}_{i \in I}$ be a multi-center. For $i \in I$, we put $L_i = M_i +(a_i)$, an ideal of $A$. Let $\bbN _I $ be the monoid $\bigoplus _{i \in I} \bbN $. If $\nu = (\nu_1 , \ldots , \nu_i, \ldots ) \in \bbN _I $ we put $L^{\nu}= L_1 ^{\nu_1} \cdots L_i ^{\nu _i} \cdots  $ (product of ideals of $A$) and $a^{\nu}= a_1^{\nu_1 } \cdots a_i ^{\nu _i} \cdots$ (product of elements of $A$). We also put $a^{\bbN _I} = \{ a^\nu | \nu \in \bbN _I \}$. 

\depr[{\cite{Ma23d}}] \label{generaldila} The dilatation of $A$ with multi-center $\{[M_i , a_i] \}_{i \in I}$ is the unital commutative ring $A[\big\{ \frac{M_i}{a_i}\big\}_{i \in I}]$ defined as follows:

$\bullet$  The underlying set of $A[\big\{ \frac{M_i}{a_i}\big\}_{i \in I}]$ is the set of equivalence classes of symbols $\frac{m}{a^{\nu}} $ where $ \nu \in \bbN _I$ and $m \in L^{\nu}$ under the equivalence relation 
\[ \frac{m}{a^\nu} \equiv \frac{p}{a^{\lambda}} \Leftrightarrow \exists \beta \in \bbN _I \text{ such that } ma^{\beta + \lambda }= p a^{\beta + \nu } \text{ in } A.\] From now on, we abuse notation and denote a class by any of its representative $\frac{m}{a^{\nu}}$ if no confusion is likely.

$\bullet$ The addition law is given by $\frac{m}{a^{\nu}}+ \frac{p}{a^{\beta}}= \frac{m a^{\beta } + p a^{\nu}}{a^{\beta + \nu}}$.

$\bullet$ The multiplication law is given by $\frac{m}{a^{\nu}}\times  \frac{p}{a^{\beta}} = \frac{ mp}{a^{\nu + \beta }}$.

$\bullet$ The additive neutral element is $\frac{0}{1}$ and the multiplicative neutral element is $\frac{1}{1}$.
\begin{flushleft}
From now on, we also use the notation $A[\frac{M}{a}]$ to denote $A[\big\{ \frac{M_i}{a_i}\big\}_{i \in I}]$.
We have a canonical morphism of rings $A \to A[\frac{M}{a}]$ given by $x \mapsto \frac{x}{1}$.
\end{flushleft}
\xdepr

The element $\frac{x}{1}$  of $ A[\frac{M}{a}]$ will be denoted by $x$ if no confusion is likely.

\fact[{\cite{Ma23d}}] \label{remdilaring} \begin{enumerate}  \item  Let $\{N_i\}_{i \in I}$ be ideals in $A$ such that  $ N_i +(a_i) = L_i$ for all $i \in I$.  Then we have identifications of $A$-algebras $A[\big\{ \frac{M_i}{a_i}\big\}_{i \in I}]= A[\big\{ \frac{N_i}{a_i}\big\}_{i \in I}]= A[\big\{ \frac{L_i}{a_i}\big\}_{i \in I}].$
\item  Dilatations of rings generalize entirely localizations of rings. Indeed, let $A$ be a ring and let $S$ be a multiplicative subset of $A$.  Then $S^{-1} A = A[\big\{ \frac{A}{s}\big\}_{s \in S}]$. 
\item Any sub-$A$-algebra of a localization $S^{-1} A$ for a subset $S \subset A$ can be obtained as a multi-centered dilatation. 
\item  Note that we did not use substraction to define dilatations of rings. In fact Definition \ref{generaldila} makes sense for arbitrary unital commutative semirings, cf. \cite[§2]{Ma23d} or more generally for categories (e.g. monoids) cf. \cite{Ma23c}. \end{enumerate}
\xfact

This construction enjoys the following properties, cf. \cite[§2]{Ma23d}. If $\#I=1$, most results appear in \cite[\href{https://stacks.math.columbia.edu/tag/052P}{Tag 052P}]{stacks-project}.
\prop[{\cite{Ma23d}}] The following assertions hold.
 \begin{enumerate} \item As $A$-algebra, $A[\frac{M}{a}]$ is generated by $\big\{ \frac{L_i}{a_i}\big\}_{i \in I}$. Since $L_i=M_i+(a_i)$, this implies that  $A[\frac{M}{a}]$ is generated by $\big\{ \frac{M_i}{a_i}\big\}_{i \in I}$.
 \item If $A$ is a domain and $a_i \ne 0$ for all $i$, then $A[\frac{M}{a}]$ is a domain.
 \item If $A$ is reduced, then $A[\frac{M}{a}]$ is reduced.
\item  The following assertions are equivalent.\begin{enumerate}
\item  There exists $\nu \in \bbN_I$ such that $a^\nu =0$ in $A$.
\item The ring $A[\frac{M}{a}]$ equals to the zero ring.\end{enumerate}
\item   Let $\nu  $ be in $\mathbb{N} _I$. The image of $a ^\nu$ in $A[\frac{M}{a}]$ is a non-zero-divisor.
\item  Let $f:A \to B$ be a morphism of rings. Let $\{[N_i,b_i]\}_{i \in I}$ be centers of $B$ such that $f(M_i) \subset N_i$  and $f(a_i)=b_i$ for all $i\in I$.  Then we have a canonical morphism of $A$-algebras
\begin{center} $\phi: A[\big\{ \frac{M_i}{a_i}\big\}_{i \in I}] \to B[\big\{ \frac{N_i}{b_i}\big\}_{i \in I}]$ .\end{center}
\item  Let $c$ be a non-zero-divisor element in $A$. Then $\frac{c}{1}$ is a non-zero-divisor in $A[\frac{M}{a}]$.
\item Let $K \subset I$ and put $J = I \setminus K$. We have a canonical morphism of $A$-algebras
\begin{center} $ \varphi :A[\big\{ \frac{M_i}{a_i}\big\}_{i \in K}] \longrightarrow A[\big\{ \frac{M_i}{a_i}\big\}_{i \in I}] . $ \end{center} Moreover 
\begin{enumerate}
\item  if $M_i\subset  (a_i)$ for all $i \in J$, then $\varphi$ is surjective, and
\item if $a_i$ is a non-zero-divisor in $A$ for all $i \in J$, then $\varphi$ is injective.
\end{enumerate}
\item Let $K \subset I$. Then we have a canonical isomorphism of $ A[\big\{ \frac{M_i}{a_i}\big\}_{i \in K}]$-algebras
\begin{center}
$A[\big\{ \frac{M_i}{a_i}\big\}_{i \in I}] = A[\big\{ \frac{M_i}{a_i}\big\}_{i \in K}] [\big\{ \frac{A[\big\{ \frac{M_i}{a_i}\big\}_{i \in K}]\frac{M_j}{1}}{\frac{a_j}{1}}\big\}_{j \in I \setminus K}], $
\end{center}
where ${A[\big\{ \frac{M_i}{a_i}\big\}_{i \in K}]\frac{M_j}{1}}$ is the ideal of $A[\big\{ \frac{M_i}{a_i}\big\}_{i \in K}]$ generated by $\frac{M_j}{1} \subset A[\big\{ \frac{M_i}{a_i}\big\}_{i \in K}]$. 
\item Assume that $a_i=a_j=:b$ for all $i,j \in I$, then 
\begin{center}$A[\big\{ \frac{M_i}{a_i}\big\}_{i \in I}]= A[\frac{\sum _{i\in I}M_i }{b}]$. \end{center}
\item Let $K \subset I$. Assume that for all $i \in I \setminus K$, there exists $k(i) $ in $K$ such that $a_{k(i)} \in L_i \subset L_{k(i)}$.
Then we have a canonical identification\begin{center}$ A[\big\{ \frac{M_i}{a_i}\big\}_{i \in I}] = A[\big\{ \frac{M_k}{a_k}\big\}_{k \in K}]_{\{ \frac{a_i}{a_{k(i)}}\}_{i \in I \setminus K}}$. ~~~~  \end{center}
\item Let $\nu \in \bbN _I$. We have $L^\nu A[\frac{M}{a}] = a^\nu A[\frac{M}{a}]$. 
\item  (Universal property) 
If $\chi : A \to B$ is a morphism of rings such that $\chi (a_i) $ is a non-zero-divisor and generates $\chi (L_i) B$ for all $i\in I$, then there exists a unique morphism $\chi '$ of $A$-algebras $A[\big\{ \frac{M_i}{a_i}\big\}_{i \in I}] \to B$. The morphism $\chi'$ sends $\frac{l}{a^\nu} $ ($\nu \in \bbN _I , l \in L^\nu$) to the unique element $b \in B $ such that $\chi ( a^\nu) b = \chi (l)$.

\item Assume that $I = \{1 , \ldots , k \}$ is finite. Then we have a canonical identification of $A$-algebras
 \begin{center}$A[\big\{ \frac{M_i}{a_i}\big\}_{i \in I}]= A[\frac{\sum _{i \in I} ( M_i \cdot \prod _{j \in I \setminus {\{i\}}} a_j )}{a_1 \cdots a_k}] .$\end{center}

\item Write $I= \colim _{J\subset I } J $ as a filtered colimit of sets. We have a canonical identification of $A$-algebras
\begin{center}
$A[\big\{ \frac{M_i}{a_i}\big\}_{i \in I}] = \colim _{J \subset I} 
A[\big\{ \frac{M_j}{a_i}\big\}_{i \in J}] $.\end{center}
\item Let $f:A \to B$ be an $A$-algebra. Put $N_i=f(M_i) B$ and $b_i = f(a_i)$ for $i \in I$. Then $B [\big\{ \frac{N_i}{b_i}\big\}_{i \in I}]$ is the quotient of $B \otimes _{A} A[\big\{ \frac{M_i}{a_i}\big\}_{i \in I}]$ by the ideal $T_b$ of elements annihilated by some element in $b^{\mathbb{N} _I }:= \{b^\nu | \nu \in \bbN _I \}$. If moreover $f:A\to B$ is flat, then $T_b=0$ and we have a canonical isomorphism 
\begin{center}  $B [\{ \frac{N_i}{b_i} \}_{i \in I }] = B \otimes _A A[\{\frac{M_i}{a_i}\}_{i\in I}].$\end{center}
\item Let $f:R\to A$ be a morphism of rings and let $\{r_i\}_{i\in I}\subset R$. Let $R'=R[ \big\{ \frac{R}{r_i} \big\}_{i\in I}]$; this is a localization  of $R$ and hence $R\to R'$ is a flat morphism. Let $A'=A\otimes_R R'$, $M_i'=M_i\otimes_RR'\subset A'$. Then, if $a_i:=f(r_i)$, the dilatation  $
 A[\big\{ \frac{M_i}{a_i}\big\}_{i \in I}]$ is isomorphic to the $A$-subalgebra of $A'=A \otimes _R R'$ generated by $\{M_i \otimes r_i^{-1} \}_{i\in I}$ and $A$.
\end{enumerate}
\xprop

We finish with an important description of dilatations in a particular case, cf. \cite[Proposition 5.5]{Ma23d} and  \cite[\href{https://stacks.math.columbia.edu/tag/0BIQ}{Tag 0BIQ}]{stacks-project}.

\prop \label{regupoly}Let $A$ be a ring. Let $a , g_1 , \ldots , g_n$ be a $H_1$-regular sequence in $A$ (cf. \cite[\href{https://stacks.math.columbia.edu/tag/062E}{Tag 062E}]{stacks-project} for $H_1$-regularity). Let $d_1 , \ldots , d_n$ be positive integers. The dilatation algebra identifies with a quotient of a polynomial algebra as follows 
\begin{center}
$A[\frac{(g_1)}{a^{d_1}}, \ldots , \frac{(g_n)}{a^{d_n}}]=A[x_1 , \ldots , x_n] / (g_1 -a^{d_1} x_1, \ldots , g_n -a^{d_n} x_n )$. 
\end{center}
\xprop

\section{Dilatations of schemes}  \label{333}

This section is an introduction to dilatations of schemes. The main references are \cite{MRR20} and \cite{Ma23d}.  Dilatations of schemes involve operations on closed subschemes that we recall at the beginning of this section.  We suggest readers to be familiar with §\ref{blow.up.definition.sec} before reading other subsections of Section \ref{333}. Note that \cite{Ma23d} deals with general algebraic spaces. In fact most results of Section \ref{333} extend to algebraic spaces.

\subsection{Definitions}\label{blow.up.definition.sec}
Let $X$ be a scheme. 
Let $Clo (X)$ be the set of closed subschemes of $X$. Recall that $Clo(X)$ corresponds to the set of quasi-coherent ideals of $\calO _X$. Let $IQCoh(\calO _X)$ denote the set of quasi-coherent ideals of $\calO _X$. It is clear that $(IQCoh(\calO _X), +, \times , 0 , \calO _X)$ is a semiring. So we obtain a semiring structure on $Clo (X)$, usually denoted by $(Clo(X), \cap , + , X, \emptyset)$. For clarity, we now recall directly operations on $Clo(X)$ and related facts. \begin{enumerate} 
\item Given two closed subschemes $Y_1,Y_2$ given by ideals $\calJ _1 , \calJ_2$, their sum $Y_1+Y_2$ is defined as the closed subscheme given by the ideal $\calJ_1 \calJ_2$. Moreover, if  $n \in \bbN$, we denote by $nY_1$ the $n$-th multiple of $Y_1$.  The set of locally principal closed subschemes of $X$ (cf. \cite[\href{https://stacks.math.columbia.edu/tag/01WR}{Tag 01WR}]{stacks-project}), denoted $Pri(X)$, forms a submonoid of $(Clo(X),+)$. Effective Cartier divisors of $X$ \cite[\href{https://stacks.math.columbia.edu/tag/01WR}{Tag 01WR}]{stacks-project}, denoted $Car(X)$, form a submonoid of $(Pri (X),+)$. Note that $Car(X)$ is a face of $Pri(X)$.
\item  We have an other monoid structure on $Clo (X)$ given by intersection. This law is denoted by $\cap$. The operation $\cap$ corresponds to the sum of quasi-coherent sheaves of ideals. The set $Clo(X)$ endowed with $ \cap, +$ is a semiring whose neutral element for $+$ is $\emptyset$ and whose neutral element for $\cap$ is $X$.
\item  Let $C \in Car(X) $, a non-zero-divisor (for $+$) in the semiring $Clo(X)$. Let $Y,Y' \in Clo (X)$. If $C+Y$ is a closed subscheme of $C+Y'$, then $Y $ is a closed subscheme of $Y'$. Moreover if $C+Y=C+Y'$, then $Y=Y'$.
\item  Let $f:X' \to X$ be a morphism of schemes. Then $f$ induces a morphism of semirings $Clo(f) : Clo (X) \to Clo (X'), Y \mapsto Y \times _X X'$. Moreover $Clo(f) $ restricted to $(Pri (X),+)$ factors through $(Pri(X'),+)$. This morphism of monoids is denoted $Pri (f)$. In general the image of the map $Pri(f) |_{Car(X)}$ is not included in $Car (X'). $ 
\item For $Y_1,Y_2 \in Clo(X)$, we write $Y_1 \subset Y_2$ if $Y_1 $ is a closed subscheme of $Y_2$. We obtain a poset $(Clo(X), \subset)$. Let $Y_1,Y_2,Y_3 \in Clo(X)$, if $Y_1 \subset Y_2 $ and $Y_1 \subset Y_3$ then $Y_1 \subset Y_2 \cap Y_3$. Let $Y_1,Y_2 \in Clo (X)$, then $(Y_1 \cap Y_2) \subset Y_1$ and $Y_1 \subset ( Y_1+Y_2)$.
\item
Finally, if $Y=\{ Y_e \} _{e \in E}$ is a subset of $Clo (X)$ and if $\nu \in \bbN ^{E}$, then we put $Y^\nu = \{\nu _e Y_e  \}_{e \in E}$ and if moreover $\nu \in \bbN _E$, then we put $\nu Y = \sum _{e \in E } \nu _e Y_e$.
\end{enumerate}
\defi[{\cite[§2.3]{MRR20} \cite{Ma23d}}]\label{defcatreg} Let $D= \{ D_i \} _{i \in I}$ be a subset of $Clo (X)$. Let $\Sch _X ^{D\text{-reg}}$ be the full subcategory of schemes $f:T \to X $ over $X$ such that    $T \times _X D_i$ is an effective Cartier divisor of $T$ for each $i$. 
\xdefi 

 \rema Assume that $I$ is finite.
The category $\Sch _X^{D\text{-}\reg}$ 
has products (cf. \cite{MRR20}, \cite{Ma23d}), but these are not in  general   the same as in the full category of $X$-schemes: indeed, given  $X',X''$   in $\Sch _X^{D\text{-}\reg}$, it is not assured that $X'\times_X X''$ still lies in  $\Sch _X^{D\text{-}\reg}$. Here are two examples. \begin{enumerate} \item  Let $A=\bbC[x,y]/(y^2-x^3)$ be the   ring of ``the cuspidal cubic'',  $A'=\bbC[t]$ the polynomial algebra in one variable, and $f:A\to A'$ the morphism sending $x$ to $t^2$ and $y$ to $t^3$. Let $B:=A'\otimes_AA'$ and note that $(f(x)\otimes 1)\cdot(t\otimes1-1\otimes t)=0$, so that $f(x)\otimes 1$ becomes a zero divisor in $B$, while $f(x)$ is not a zero divisor in $A'$. 
\item Let $A = \bbZ [X]$ and $A'= A'' =\bbZ $. Let $f' : A \to A', X \mapsto 0$ and $f'':A \to A'' , X \mapsto 2$, two morphisms of rings. Then $2=f'(2) = f''(2)$ is a non-zero-divisor in $\bbZ$, however $2$ is a zero-divisor in $A' \otimes _A A \cong \bbZ / 2 \bbZ$. 
\end{enumerate}
 \xrema

If $T'\to T$ is flat and $T\to X$ is an object in $\Sch _X ^{D\text{-reg}}$, so is the composition $T'\to T\to X$.
In particular, the category $\Sch _X ^{D\text{-reg}}$ can be equipped with the fpqc/fppf/\'etale/Zariski Grothendieck topology so that the notion of sheaves is well-defined.

\fact[\cite{Ma23d}] \label{fact3.3}  Let $D= \{ D_i \} _{i \in I}$ be a subset of $Clo (X)$. 
 \begin{enumerate} \item Let $f:T\to X$ be an object in $\Sch _X ^{D\text{-reg}}$. Then for any $\nu \in \bbN _I$, the scheme $T \times _X \nu D$ is an effective Cartier divisor of $T$, namely $\nu ( T \times _X D)$.
\item Assume that $\#I$ is finite, then $\Sch _X ^{D\text{-reg}}$ equals $\Sch _X^{\sum _{i \in I} D_i}$.
\end{enumerate}
\xfact

\defi[\cite{Ma23d}] \label{multi}
 A multi-center in $X$ is a set $\{ [Y_i , D_i ]\} _{i \in I} $ such that \begin{enumerate}
 \item  $Y_i $ and $ D_i$ belong to $ Clo (X) $,
 \item there exists an affine open covering $\{ U_{\gamma } \to X \}_{\gamma \in \Gamma} $ of $X$ such that $D_i |_{U_\gamma}$ is principal for all $i \in I $ and $\gamma \in \Gamma $ (in particular $D_i $ belongs to $Pri (X)$ for all $i$). \end{enumerate}
 In other words a multi-center  $\{ [Y_i , D_i ]\} _{i \in I} $ is a set of pairs of closed subschemes such that locally each $D_i$ is principal.
\xdefi

\rema Let $\{ Y_i , D_i \} _{i \in I} $ be such that   $Y_i \in Clo (X) $ and $D_i \in Pri (X)$ for any $i \in I$. Assume that $I$ is finite, then $\{ [Y_i , D_i ]\} _{i \in I} $ is a multi-center in $X$, i.e. the second condition in Definition \ref{multi} is satisfied.
\xrema 

We now fix a multi-center $\{ [Y_i , D_i ]\} _{i \in I} $ in $X$.
 Denote by
$\calM_i$, respectively $\calJ_i$,  the quasi-coherent sheaf
of ideals of $\calO_X$ defining   $Y_i$, respectively $D_i$. We put $Z_i = Y_i \cap D_i$ and $\calL_i = \calM _i + \calJ _i $ so that $Z_i$ is defined by $\calL _i$. We put $Y= \{Y_i \}_{i \in I}$, $D= \{D_i \}_{i \in I}$ and $Z= \{ Z_i \}_{i \in I}$.
We now introduce dilatations $\calO _X$-algebras by glueing. 

\depr \label{defiqcoalg} \begin{sloppypar}
The dilatation of $\calO_X$ with multi-center $\{[\calM _i , \calJ _i ] \}_{i \in I}$ is the quasi-coherent $\calO _X$-algebra  $\calO _X \Big[\Big\{ \frac{\calM_i}{\calJ_i}\Big\}_{i \in I} \Big]$ obtained by glueing as follows.  The quasi-coherent $\calO _X$-algebra $\calO _X \Big[\Big\{ \frac{\calM_i}{\calJ_i}\Big\}_{i \in I} \Big]$ is characterized by the fact that its restriction, on any open subscheme $ U \subset X$ such that $U$ is an affine scheme and each $D_i$ is principal on $U$ and generated by $a_{iU}$, is given by
\[ \Big( \calO _X \Big[\Big\{ \frac{\calM_i}{\calJ_i}\Big\}_{i \in I} \Big] \Big) _{\big|_{U_{}}}  =\Big( \Gamma (U , \calO _X ) \Big[\Big\{ \frac{\Gamma (U, \calM _i) }{a_{i U} }\Big\}_{i \in I} \Big] \Big) ^{\widetilde{~~~}} \]
where $(~)\widetilde{~~}$ denotes the associated sheaf of algebras on $U$. \end{sloppypar}
\xdepr

\defi[{\cite{Ma23d}}] \label{defimultidilaalgsp}
The \textit{dilatation} of $X$  with multi-center $\{[Y_i , D_i ]\} _{i \in I} $ is the $X$-affine scheme 
\[
\Bl_Y^DX\defined \Spec _X \big( \calO _X \Big[\Big\{ \frac{\calM_i}{\calJ_i}\Big\}_{i \in I} \Big] \big).
\]
The terminologies \textit{affine blowups} and \textit{affine modifications} are also used. \xdefi

\rema
In the mono-centered case, this definition is the one of \cite{MRR20}. If moreover $D$ is a Cartier divisor, then one has another equivalent definition (cf. Proposition \ref{blow.up.closed.in.cone.lemm}) that goes back to \cite{KZ99} and \cite{Du05}.
\xrema 

\rema \label{remaYZYZYZ}
We always have $\Bl _Y ^D X = \Bl _Z ^D X$.
\xrema

\nota \label{notabl}We will also use the notation  $ \Bl \big\{ _{Y_i}^{D_i } \big\}_{i\in I} X$ and $\Bl _{ \{Y_i\}_{i\in I}}^{\{D_i\}_{i\in I}} X $ to denote $\Bl _Y ^D X $. If $I= \{i\}$ is a singleton, then we also use the notation $\Bl _{Y_i} ^{D_i} X $. If $K \subset I$, then we sometimes use the notation $\Bl _{ \{Y_i\}_{i\in K}, \{Y_i\}_{i\in I \setminus K}}^{\{D_i\}_{i\in K},\{D_i\}_{i\in I \setminus K} } X $. If $I = \{1 , \ldots ,k\}$, then we use the notation $\Bl _{Y_1, \ldots , Y_k}^{D_1 , \ldots , D_k} X$. Etc. \xnota 

\defi \label{dilamapdee}We say that a morphism $f:X' \to X$ is a dilatation morphism if $f  $ is equal to $ \Bl \big\{ _{Y_i}^{D_i } \big\}_{i\in I} X \to X$ for some multi-center $\{[Y_i , D_i]\}_{i \in I}.$\xdefi

\fact[\cite{Du05, MRR20, Ma23d}] If $\#I$ is finite and $Y_i = \emptyset $ for all $i$, then the dilatation morphism $ \Bl \big\{ _{\emptyset}^{D_i } \big\}_{i\in I} X \to X$ is an open immersion. 
\xfact

\prop[{\cite{Ma23d}}] \label{defined} Put $f : \Bl_Y ^D X \to X$. Then the morphism of monoids $Clo (f) |_{Car(X)} $ factors through $Car( \Bl_Y ^D X )$. In other words, any effective Cartier divisor $C \subset X$ is defined for $f$ (cf. \cite[\href{https://stacks.math.columbia.edu/tag/01WV}{Tag 01WV}]{stacks-project}), i.e. the fiber product $C \times _X \Bl_{Y}^D X \subset \Bl_{Y}^D X $ is an effective cartier divisor.\xprop

\subsection{Exceptional divisors}

We proceed with the notation from \S\ref{blow.up.definition.sec}.

\prop[{\cite{MRR20, Ma23d}}] \label{blow.up.Cartier.lemm}
As closed subschemes of $\Bl_Y^DX$, one has, for all $\nu \in \mathbb{N}_I$,
\[
\Bl_Y^DX\x_X \nu Z=\Bl_Y^DX\x_X \nu D,
\]
which is an effective Cartier divisor on $\Bl_Y^DX$.
\xprop

\subsection{Relation to affine projecting cone}

We proceed with the notation from \S\ref{blow.up.definition.sec}. We assume that the set $\{D_i\} _{i \in I}$ is included in $Car(X)$. In this case, we can also realize $\Bl_Y^DX$ as a
closed subscheme of the multi-centered affine projecting cone associated to $X,Z$ and $D$. 

\defi 
The {\em affine projecting cone  $\calO_X$-algebra } with multi-center $\{[Z_i=V(\calL_i) , D_i = V (\calJ _i)]\}_{i \in I} $ is
\[
\Cone_{\calL}^{\calJ} \calO_ X\defined\bigoplus_{\nu \in \mathbb{N} _I} \calL^\nu \otimes\calJ^{-\nu}.
\]
The {\em affine projecting cone }  of $X$ with multi-center $\{[Z_i , D_i]\}_{i \in I} $ is \[ \Cone _Z^D X \defined \Spec \big( \Cone_{\calL}^{\calJ} \calO_ X \big).\]
\xdefi

\prop[{\cite{KZ99, Du05, MRR20, Ma23d}}] \label{blow.up.closed.in.cone.lemm} The dilatation $\Bl_Z^DX$ is the closed
subscheme of the affine projecting cone $\Cone_Z^DX$ defined by the equations $\{\varrho_i-1\}_{i\in I} $,
where  for all $i \in I$, $\varrho _i\in \Cone_{\calL}^{\calJ} \calO_ X$ is the image of $1 \in \calO _X$ under the map
\[\calO_X \cong\calJ _i\otimes\calJ _i ^{-1}\subset \calL _i\otimes\calJ _i^{-1} \subset \Cone_{\calL}^{\calJ} \calO_ X. \]
\xprop

\subsection{Description of the exceptional divisor in the mono-centered case}
We proceed with the notation from \S\ref{blow.up.definition.sec}. We assume $I= \{i\}$ is a singleton and we omit the subscripts $i$ in notation.
We saw in Proposition~\ref{blow.up.Cartier.lemm} that the preimage of the center
$\Bl_Z^DX\x_XZ=\Bl_Z^DX\x_XD$ is an effective Cartier divisor in $\Bl_Z^DX$.
In order to describe it following \cite{MRR20}, as before we denote by $\calL$ and $\calJ$
the sheaves of ideals of $Z$ and $D$ in $\calO_X$. Also we let
$\calC_{Z/D}=\calL/(\calL^2+\calJ)$ and $\calN_{Z/D}=\calC_{Z/D}^\vee$
be the conormal and normal sheaves of $Z$ in~$D$.

\prop[{\cite[Proposition 2.9]{MRR20}}] \label{exceptional divisor} 
Assume that $D\subset X$ is an effective Cartier divisor and that $Z\subset D$ is a regular immersion. Write $\calJ_Z:=\calJ|_Z$.
\begin{enumerate}
\item[(1)] The exceptional divisor $\Bl_Z^DX\x_XZ\to Z$ is an
affine bundle (i.e. a torsor under a vector bundle), Zariski
locally over $Z$ isomorphic
to
$\bbV(\calC_{Z/D}\otimes \calJ_Z^{-1})\to Z$.
\item[(2)] If $H^1(Z,\calN_{Z/D}\otimes \calJ_Z)=0$ (for example if $Z$ is affine),
then $\Bl_Z^DX\x_XZ\to Z$ is {\em globally} isomorphic to
$\bbV(\calC_{Z/D}\otimes \calJ_Z^{-1})\to Z$.
\item[(3)] If $Z$ is a transversal intersection in the sense that there
is a cartesian square of closed subschemes whose vertical maps are
regular immersions
\[
\begin{tikzcd}[column sep=30]
W \ar[r,hook] \arrow[rd, "\square",phantom] & X \\
Z \ar[r,hook] \ar[u,hook] & D \ar[u,hook]
\end{tikzcd}
\]
then $\Bl_Z^DX\x_XZ\to Z$ is {\em globally and canonically} isomorphic
to $\bbV(\calC_{Z/D}\otimes \calJ_Z^{-1})\to Z$.
\end{enumerate}
\xprop

\subsection{Universal property in the absolute setting}\label{blow.up.univ.ppty.section}
We proceed with the notation from \S\ref{blow.up.definition.sec}.
As $\Bl_Y^DX\to X$ defines an object in $\Sch_X^{D\text{-}\reg}$ by Proposition \ref{blow.up.Cartier.lemm}, the contravariant functor
\begin{equation}\label{blow.up.represent.eq}
\Sch_X^{D \text{-}\reg}\to \Sets, \;\;\;(T\to X)\mapsto \Hom_{X\text{-Schemes}}\big(T,\Bl_Y^DX\big)
\end{equation}
together with $\id_{\Bl_Y^DX}$ determines $\Bl_Y^DX\to X$ uniquely up to unique isomorphism.
The next proposition gives the universal property of dilatations in the absolute setting. 

\prop[{\cite{MRR20, Ma23d}}] \label{blow.up.rep.prop}
The dilatation $\Bl_Y^DX\to X$ represents the contravariant functor $\Sch_X^{D\text{-}\reg}\to \Sets$ given by
\begin{equation}\label{blow.up.iso.eq}
(f\co T\to X) \;\longmapsto\; \begin{cases}\{*\}, \; \text{if $f|_{T\x_XD_i}$ factors through $Y_i\subset X$ for $i \in I$;}\\ \varnothing,\;\text{else.}\end{cases}
\end{equation}
\xprop
Note that here $\{*\}$ denotes a singleton. 
\subsection{Universal property in the relative setting} 

In the relative setting, Proposition \ref{blow.up.rep.prop} implies the following statement. Let $S$ be a scheme and let $X$ be a scheme over $S$. Let $C=\{C_i\}_{i\in I}$ be closed subschemes of $S$ such that, locally, each $C_i$ is principal. Put $D= \{ C_i \times _S X\}_{i \in I}$. Let $Y=\{Y_i\}_{i\in I}$ be closed $S$-subspaces of $X$. We put $\Bl_{Y}^{C^{}} X:= \Bl_{Y}^{D^{}} X$. 
  
\fact[\cite{Ma23d}] \label{factintrro} $\Bl _Y ^{C}X$ represents the functor from $Sch ^{C\text{-reg}}_S$ to $Set$ given by
\[( f: T \to S ) \mapsto \{x \in \Hom _S (T,X) |~ {x_|}_{C_i^{}} : T \times _S C_i^{} \to X \times _S C_i^{} \text{ factors through } Y_i \times _S C_i \}. \]
\xfact 

\rema[\cite{Ma23d}] \label{remainj} Fact \ref{factintrro} implies that for any $ T \in Sch ^{C\text{-reg}}_S$ (e.g. $T=S$ if each $C_i$ is a Cartier divisor in $S$) we have a canonical inclusion on $T$-points $ \Bl_Y ^{C^{}} X (T ) \subset X(T)$. But in general $\Bl_Y ^{C^{}} X \to X$ is not a monomorphism in the full category of $S$-schemes.
\xrema

\subsection{Dilatations or affine blowups }

We proceed with the notation from \S\ref{blow.up.definition.sec}. We assume $I= \{i\}$ is a singleton and we omit the subscripts $i$ in notation.

\prop[ \cite{MRR20}] \label{blow.up.open.lemm}
The dilatation $\Bl_Z^DX$ is the open subscheme of the blowup $\Bl_ZX=\Proj(\Bl_\calI\calO_X)$ defined by the complement of $V_+(\overline\calJ)$ where $\overline\calJ$
is the sheaf of ideals generated by $\calJ\subset\calI$ and $\calI$
is the degree 1 part of $\Bl_\calI\calO_X$.
\xprop

For this reason dilatations are also called affine blowups.
 A similar description holds in the multi-centered case, cf \cite{Ma23d}.

\subsection{Combinatorial and arithmetic  relations}
We proceed with the notation from \S\ref{blow.up.definition.sec}.

\prop[{\cite{Ma23d}}] \label{multietape} Let $J $ be a subset of $I$ and put $K = I \setminus J$. Then
\[ \Bl \big\{_{Y_i} ^{D_i}\big\}_{i \in I} X = \Bl \Big\{_{Y_k \times _X  \Bl \big\{_{Y_i}^{D_i}\big\}_{i \in J} X } ^{
D_k \times _X  \Bl \big\{_{Y_i}^{D_i}\big\}_{i \in J} X  }\Big\} _{k \in K } \Bl \big\{_{Y_i}^{D_i}\big\}_{i \in J} X  .\] In particular, there is a unique $X$-morphism \[\Bl \big\{_{Y_i}^{D_i}\big\}_{i \in I} X  \to \Bl \big\{_{Y_i}^{D_i}\big\}_{i \in J} X . \]
\xprop 

\prop[{\cite{ Ma23d}}]
 Let $K \subset I$ be such that \begin{enumerate} \item  $I \setminus K $ is finite, \item for all $i \in I \setminus K $, there exists $k(i) \in K $ such that $Z_{k(i)} \subset Z_i \subset D_{k(i)}$. \end{enumerate} Then the canonical morphism given by Proposition \ref{multietape}
 \[ \Bl \big\{_{Y_i}^{D_i}\big\}_{i \in I} X  \to \Bl \big\{_{Y_i}^{D_i}\big\}_{i \in K} X  \] is an open immersion.
\xprop

\prop[{\cite{ Ma23d}}] \label{colimitspacegeneral} Write $I = \colim _{J \subset I} J $ as a filtered colimit of sets where transition maps are given by inclusions of subsets. We have a canonical identification 
\begin{center}$\Bl \big\{^{D_i}_{Y_i}  \big\}_{i \in I}  X = \lim_{J \subset I} \Bl \big\{^{D_i}_{Y_i}  \big\}_{i \in J} X $.\end{center} 
\xprop 

\prop[{\cite{ Ma23d}}] \label{multisingle} 
Assume that $\#I=k$ is finite. We fix an arbitrary bijection $I = \{ 1 ,\ldots , k\}$.  We have a canonical isomorphism of $X$-schemes
\[ \Bl _{\{Y_i\}_{i \in I}}^{\{D_i\}_{i \in I} }X  \cong \Bl ^{(\Bl\cdots ) \times _X D_k}_{( \Bl \cdots ) \times _X Y_k }\Biggl(\cdots \Bl_{(\Bl \cdots ) \times _X Y_3}^{(\Bl\cdots ) \times _X D_3}\biggl( \Bl _{(\Bl_{Y_1}^{D_1}X) \times _{X} Y_2} ^{(\Bl_{Y_1}^{D_1}X) \times _{X} D_2} \bigl(\Bl _{Y_1}^{D_1} X \bigl) \biggl) \Biggl). \]
\xprop

\prop[{\cite{ Ma23d}}] \label{formulamultimono} 
Assume that $\#I=k$ is finite. We fix an arbitrary bijection $I = \{ 1 ,\ldots , k\}$. We have a canonical isomorphism of $X$-schemes
 \[\Bl _{\{Y_i\}_{i \in I}}^{\{D_i\}_{i \in I} }X  \cong \Bl _{\bigcap _{i \in I }( Y_i + D_1+ \ldots + D_{i-1} + D_{i+1} + \ldots + D_k)} ^{D_1+ \ldots + D_k} X.\]
\xprop

\subsection{Functoriality} \label{blow.up.base.functoriality.sec} We proceed with the notation from \S\ref{blow.up.definition.sec}.
Let $X'$ and $\{[Y'_i , D'_i]\} _{i\in I}  $ be another datum as in \S\ref{blow.up.definition.sec}. As usual, put $Z'_i= Y_i' \cap D'_i$.
A morphism $f:X'\to X$ such that, for all $i \in I$, its restriction to $D'_i$ (resp.~${Z'_i}$) factors through $D_i$ (resp.~$Z_i$), and such that $f^{-1} (D_i) = D_i'$, induces a unique morphism  $\Bl_{Y'}^{D'}X'\to \Bl_Y^DX$ such that the following diagram of schemes commutes
\[
\begin{tikzpicture}[baseline=(current  bounding  box.center)]
\matrix(a)[matrix of math nodes, 
row sep=1.5em, column sep=2em, 
text height=1.5ex, text depth=0.45ex] 
{ 
\Bl_{Y'}^{D'}X'& \Bl_Y^DX \\ 
X'& X.\\}; 
\path[->](a-1-1) edge node[above] {} (a-1-2);
\path[->](a-2-1) edge node[above] {} (a-2-2);
\path[->](a-1-1) edge node[right] {} (a-2-1);
\path[->](a-1-2) edge node[right] {} (a-2-2);
\end{tikzpicture}
\]

\subsection{Base change} \label{blow.up.base.change.sec}
We proceed with the notation from \S\ref{blow.up.definition.sec}.
Let $X'\to X$ be a map of schemes, and denote by $Y'_i, Z'_i, D'_i \subset X'$ the preimages of $Y_i,Z_i,D_i \subset X$ respectively.
Then $D'_i \subset X'$ is locally principal for any $i$ so that the dilatation $\Bl_{Y'}^{D'}X'\to X'$ is well-defined. 
By \S\ref{blow.up.base.functoriality.sec} there is a canonical morphism of $X'$-schemes
\begin{equation}\label{blow.up.base.change.eq}
\Bl_{Y'}^{D'}X'\;\longto\; \Bl_Y^DX\x_{X}X'.
\end{equation}

\lemm[{\cite{MRR20, Ma23d}}]\label{blow.up.base.change.lemm}
If $\Bl_Y^DX\x_{X}X'\to X'$ is an object of $\Sch_{X'}^{D\text{-}\reg}$, then \eqref{blow.up.base.change.eq} is an isomorphism.
\xlemm

\coro[{\cite{MRR20, Ma23d}}] \label{blow.up.base.change.cor}
If a morphism $X'\to X$ is flat and satisfies a property $\calP$ which is
stable under base change, then $\Bl_{Y'}^{D'}X'\to\Bl_{Y}^{D}X$ is flat
and satisfies $\calP$.
\xcoro

\subsection{Iterated multi-centered dilatations} \label{sec4}
We proceed with the notation from § \ref{blow.up.definition.sec}. Let $\nu , \theta \in \bbN ^I$ such that $\theta \leq \nu$, i.e. $\theta _i \leq \nu _i $ for all $i\in I$.

\prop[{\cite{Ma23d}}] \label{prop1iter}
There is a unique $X$-morphism 
\[ \varphi _{\nu , \theta}: \Bl _{Y}^{D^{\nu}} X \longrightarrow \Bl _{Y}^{D^{\theta}} X  . \]
\xprop

Assume now moreover that $\nu , \theta \in \bbN _I \subset \bbN ^I.$
We will prove that, under some assumptions, $\varphi _{\nu , \theta}$ is a dilatation morphism with explicit descriptions. We need the following observation.

\prop[{\cite{MRR20, Ma23d}}] \label{liftclo}  Assume that we have a commutative diagram of schemes \[
\begin{tikzcd}B \ar[rr, "f'"] \ar[dr,"f"]  & & \Bl_Y^D X \ar[dl, ] \\  &X \end{tikzcd}\] where the right-hand side morphism is the dilatation map.
 Assume that $f$ is a closed immersion. Then $f'$ is a closed immersion.
\xprop

We now assume that $Z_i \subset Y_i$ is a Cartier divisor inclusion for all $i \in I$.
 Let $\mathcal{D}_i $ be the canonical diagram of closed immersions \[\begin{tikzcd} Y_i \ar[r] \arrow[rd, "\square",phantom] &  \Bl _{Y_i}^{\nu_i D_i } X   \\ Z_i \ar[u] \ar[r] & \ar[u] D_i \end{tikzcd} \] obtained by Proposition \ref{liftclo}.
Let $f_i$ be the canonical morphism (e.g. cf.  \ref{multietape} or \ref{prop1iter})
\[ \Bl _{Y}^{D^{\nu}} X \to \Bl _{Y_i}^{{\nu_i}D_i } X  .\] 
 We denote by $Y_i \times _{\Bl _{Y_i}^{{\nu_i}D_i } X}\Bl _{Y}^{D^{\nu}}X$ the fiber product obtained via the arrows given by $f_i$ and $\mathcal{D}_i$. We use similarly the notation $D_i \times _{\Bl _{Y_i}^{{\nu_i}D_i } X}\Bl _{Y}^{D^{\nu}}X$.

\prop[{\cite[§7.2]{PY06} \cite{ MRR20, Ma23d}}] \label{nutheta} Recall that $\theta \leq \nu $. Put $\gamma = \nu - \theta$. Put $K = \{ i \in I | \gamma _i >0 \}.$ We have an identification
 \[\Bl _{Y}^{D^{\nu}} X = \Bl ^{\{{{\gamma _i}D_i \times _{\Bl _{Y_i}^{\theta _i D_i } X}\Bl _{Y}^{D^{\theta}} X\}_{i \in K}}} _{\{Y_i   \times _{\Bl _{Y_i}^{\theta _i D_i } X}\Bl _{Y}^{D^{\theta}} X \}_{i \in K}} \Bl _{Y}^{D^{\theta}} X .\]
  In particular the unique $X$-morphism 
\[ \varphi _{\nu , \theta}: \Bl _{Y}^{D^{\nu}} X \longrightarrow \Bl _{Y}^{D^{\theta}} X   \] of Proposition \ref{prop1iter} is a dilatation map.
\xprop

It is now natural to introduce the following terminology.

\defi 
 For any $\nu \in \bbN ^k, $ let us consider
\[ \Bl _{Y}^{D^{\nu}} X= \Bl \big\{_{Y_i}^{\nu _i D_i}\big\}_{i \in I} X\] and call it the $\nu$-th iterated dilatation of $X$ with multi-center $\{Y_i,D_i\}_{i\in I}$.
\xdefi

\subsection{Some flatness and smoothness results}
\label{sec6}

We proceed with the notation from \S\ref{blow.up.definition.sec} and assume $I= \{i\}$ is a singleton and we omit the subscripts $i$ in notation.
We assume further that there exists a scheme $S$ under $X$ together with a locally principal closed subscheme $S_0\subset S$ fitting into a commutative diagram of schemes
\begin{equation}\label{blow.up.smoothness.diag}
\begin{tikzpicture}[baseline=(current  bounding  box.center)]
\matrix(a)[matrix of math nodes, 
row sep=1.5em, column sep=2em, 
text height=1.5ex, text depth=0.45ex] 
{ 
Z& D& X \\ 
& S_0 & S,\\}; 
\path[->](a-1-1) edge node[above] {} (a-1-2);
\path[->](a-1-2) edge node[above] {} (a-1-3);
\path[->](a-2-2) edge node[above] {} (a-2-3);
\path[->](a-1-1) edge node[right] {} (a-2-2);
\path[->](a-1-2) edge node[right] {} (a-2-2);
\path[->](a-1-3) edge node[right] {} (a-2-3);
\end{tikzpicture}
\end{equation}
where the square is cartesian, that is $D\to X_0:=X\times_SS_0$
is an isomorphism.

\prop[\cite{MRR20}] \label{blow.up.smoothness.prop}
Assume that $S_0$ is an effective Cartier divisor on $S$.
\begin{enumerate}
\item[(1)] If $Z\subset D$ is regular, then
$\Bl_Z^DX\to X$ is of finite presentation.
\item[(2)] If $Z\subset D$ is regular, the fibers of
$\Bl_Z^DX\times_SS_0\to S_0$ are connected \textup{(}resp.
irreducible, geometrically connected, geometrically irreducible\textup{)}
if and only if the fibers of $Z\to S_0$ are.
\item[(3)] If $X\to S$ is flat and if moreover one of the following holds:
\begin{itemize}
\item[\rm (i)] $Z\subset D$ is regular, $Z\to S_0$ is flat
and $S,X$ are locally noetherian,
\item[\rm (ii)] $Z\subset D$ is regular, $Z\to S_0$ is flat and
$X\to S$ is locally of finite presentation,
\item[\rm (iii)] the local rings of $S$ are valuation rings,
\end{itemize}
then $\Bl_Z^DX\to S$ is flat.
\item[(4)] If both $X\to S$, $Z\to S_0$ are smooth, then $\Bl_Z^DX\to S$ is smooth.
\end{enumerate}
\xprop

\rema Complementary smoothness and flatness results for multi-centered dilatations can be found in \cite[§6]{Ma23d}.
\xrema

\subsection{Remarks}

Dilatations commute with algebraic attractors \cite[Proposition 13.1]{Ma23a}.

\section{Dilatations of group schemes or Néron blowups} \label{neron} 
One of the key properties allowed by dilatations is that it preserves the structure of group schemes in many cases. Dilatations of group schemes are also called N\'eron blowups and we also often use this terminology.

\subsection{Definitions of multi-centered Néron blowups} \label{subnermulti}
Let $S$ be a scheme and $G\to S$ a group scheme.
Let $C=\{C_i\}_{i\in I} $ be a set of locally principal closed subschemes of $S$.  Put  $ G|_{C_i} = G \times _S C_i$ and ${G}|_C= \{ G|_{C_i} \}_{i \in I}$.
Let $H_i\subset G|_{C_i}$ be a closed subgroup scheme over $C_i$ for all $i \in I$ and let $H= \{H_i \}$. Then $\{[H_i,{G}|_{C_i}]\}_{i \in I}$ is a multi-center which is denoted as $[H,G|_C]$. The multi-centered dilatation 
\[
\calG:=\Bl_H^{{G}|_C}G\longrightarrow G\] is called the {\it N\'eron blowup} of $G$ with multi-center $[H,G|_C]$. We also use the notation $\Bl_H^{C}G$ to denote $\calG$.  
In the case $I$ has a single element, we shall refer to $\Bl_H^{C}G$ as mono-centred N\'eron blowups.
By Proposition \ref{blow.up.Cartier.lemm} the structural morphism $\calG\to S$
defines an object in $\Sch_S^{C\text{-}\reg}$.

\prop [{\cite{MRR20, Ma23d}}] \label{Neron.blow.lemm} 
Let $\calG\to S$ be the above multi-centered N\'eron blowup.
\begin{enumerate}
\item[(1)]   The $S$-scheme $\calG$ represents the contravariant functor
$\Sch_S^{C\text{-}\reg}\to  \Sets$ 
\[
T \,\,\longmapsto\,\,\left\{T\to G\,\,:\,\,\begin{array}{c}\text{$T|_{C_i}\to G|_{C_i}$ factors through} \\ \text{$H_i\subset G|_{C_i}$ for all $i$}\end{array}\right\}.
\]

\item[(2)] Let $T \to S$ be an object in $\Sch_S^{C\text{-}\reg}$, then as subsets of $G(T)$ \[ \calG (T) = \bigcap_{i \in I} \big( \Bl_{H_i}^{C_i}G \big)(T).\]

\item[(3)] The map $\calG\to G$ is affine. 
Its restriction over $C_i$ factors as $\calG|_{C_i}\to H_i \subset  G|_{C_i}$ for all $i$.

\item[(4)] If the N\'eron blowup $\calG\to S$ is flat, then it is equipped with the structure of a group scheme such that $\calG\to G$ is a morphism of $S$-group schemes.
\end{enumerate}
\xprop 

\rema We saw that in favorable cases, dilatations preserve group scheme structures. In fact dilatations preserve similarly monoid scheme structures and Lie algebra scheme structures, or more generally structures defined by products, cf. \cite[§7]{Ma23d} for details.
\xrema 

\rema Dilatations commute with the formation of Lie algebra schemes in a natural sense \[ \bbL ie (\Bl_H^{G|_C}G ) \cong \Bl \big\{{}^{\bbL ie (G) \times _S C_i}_{\bbL ie ( H_i )} \big\}_{i \in I} \bbL ie (G) \] cf. \cite[§7]{Ma23d} for precise flatness assumptions.
\xrema 
\prop[\cite{Ma23d}]
Assume that $C_i$ is a Cartier divisor in $S$ for all $i$. Assume that $G \to S$ is a flat group scheme. Let $\eta : K\to G$ be a morphism of group schemes over $S$ such that $K \to S$ is flat. Assume that $H_i  \subset G $ is a closed subgroup scheme over $S$ such that $H_i \to S$ is flat for all $i$. Assume that $\Bl_{H}^{C} G \to S$ is flat (and in particular a group scheme). Assume that, for all $i$,
 $K_{C_i} $ commutes with ${H_i}_{C_i}$ in the sense that the morphism $K_{C_i} \times _{C_i} {H_i}_{C_i} \to G_{C_i} $, $(k,h) \mapsto \eta(k)h\eta(k)^{-1} $ equals the composition morphism $K_{C_i} \times _{C_i} {H_i}_{C_i} \to {H_i}_{C_i} \subset G_{C_i}$, $(k,h) \mapsto h$. Then $K$ normalizes $\Bl _H^C G$, more precisely the solid composition map \[
 \begin{tikzcd} K \times _S \Bl _{H}^C G \ar[rr, "Id \times \Bl "] \ar[rrd, dashrightarrow] & & K \times _S G \ar[rrr, "k {,} g \mapsto \eta(k)g\eta(k)^{-1} "] & & & G\\ &  &\Bl _H ^C G  \ar[rrru, "\Bl"]&  & \end{tikzcd} \]
 factors uniquely through $\Bl _H^C G $.
 \xprop

 \subsection{Mono-centered Néron blowups}
 We proceed with the notation from §\ref{subnermulti}. We now deal with the mono-centered case, i.e. $k=1$. We put $S_0= C_1$ and $H=H_1$.

 \theo[\cite{WW80, MRR20}] \label{theo:dilatations of subgroups}
 Assume that 
$G\to S$ is flat, locally finitely presented and $H\to S_0$ is flat,
regularly immersed in $G_0 $. Let $\calG\to G$ be the dilatation
$\Bl_H^{S_0} G$ with exceptional divisor $\calG_0:=\calG\x_SS_0$.
Let $\calJ$ be the ideal sheaf of $G_0$ in $G$ and
$\calJ_H:=\calJ|_H$. Let $V$ be the restriction of the normal
bundle $\bbV(\calC_{H/G_0}\otimes \calJ_H^{-1})\to H$ along the
unit section $e_0\co S_0\to H$.
\begin{enumerate}
\item[(1)] Locally over $S_0$, there is an exact sequence of
$S_0$-group schemes $1 \to V \to \calG_0 \to H\to 1$.
\item[(2)] Assume that we have a lifting of $H$ to a flat $S$-subgroup
scheme of~$G$. Then there is globally an exact, canonically
split sequence $1 \to V \to \calG_0 \to H\to 1$.
\item[(3)] If $G\to S$ is smooth, separated and $\calG\to G$ is the
dilatation of the unit section of $G$, there is a canonical
isomorphism of smooth $S_0$-group schemes
$\calG_0 \isomto \Lie(G_0/S_0)\otimes\Norm_{S_0/S}^{-1}$
where $\Norm_{S_0/S}$ is the normal bundle of $S_0$ in $S$.
\end{enumerate}
\xtheo

\rema
In the situation of Theorem \ref{theo:dilatations of subgroups} (2),
the group $H$ acts by conjugation on
$V=\bbV(e_0^*\calC_{H/G_0}\otimes \calJ_{S_0}^{-1})$.
It is expected that this additive action is linear, and is
in fact none other than the ``adjoint'' representation of $H$ on
its normal bundle as in \cite[Exp.~I, Prop.~6.8.6]{DG70}.
When the base scheme is the spectrum of a discrete
valuation ring this is proved in \cite[Prop.~2.7]{duong-hai-dos_santos18}. 
\xrema
 
Assume now that $j\co S_0\hookto S$ is an effective Cartier divisor,
that $G\to S$ is a flat, locally finitely presented group scheme
and that $H\subset G_0:=G\times_SS_0$ is a flat, locally finitely
presented closed
$S_0$-subgroup scheme. In this context, there is another
viewpoint on the dilatation $\calG$ of $G$ in $H$, namely as the
kernel of a certain map of syntomic sheaves.

To explain this, let $f\co G_0\to G_0/H$ be the morphism to
the fppf quotient sheaf, which by Artin's theorem
(\cite[Cor.~6.3]{Ar74} and
\cite[\href{https://stacks.math.columbia.edu/tag/04S6}{04S6}]{stacks-project})
is representable by an algebraic
space. By the structure theorem for algebraic group schemes
(see \cite[Exp.~VII$_{\on{B}}$, Cor.~5.5.1]{DG70})
the morphisms $G\to S$ and $H\to S_0$ are syntomic. Since
$f\co G_0\to G_0/H$ makes $G_0$ an $H$-torsor, it follows
that $f$ is syntomic as well.

\prop[{\cite[Lemma 3.8]{MRR20}}] \label{lemma:short exact sequence}
Let $S_{\syn}$ be the small syntomic site of $S$. Let
$\eta\co G\to j_*j^*G$ be the adjunction map in the category
of sheaves on $S_{\syn}$ and consider the composition
$v=(j_*f)\circ \eta$:
\[
\begin{tikzcd}
G \ar[r,"\eta"] & j_*j^*G=j_*G_0 \ar[r,"j_*f"] & j_*(G_0/H).
\end{tikzcd}
\]
Then the dilatation $\calG\to G$ is the kernel of $v$.
More precisely, we have an exact sequence of sheaves of
pointed sets in $S_{\syn}$:
\[
\begin{tikzcd}[column sep=15]
1 \ar[r] & \calG \ar[r] & G \ar[r,"v"] & j_*(G_0/H) \ar[r] & 1.
\end{tikzcd}
\]
If $G\to S$ and $H\to S_0$ are smooth, then the sequence is
exact as a sequence of sheaves on the small \'etale site of $S$.
\xprop

As a corollary, one has the useful and typical result as follows.
\coro[\cite{MRR20}] \label{lemma:congruence}
Let $\calO$ be a ring and $\pi\subset \calO$ an invertible ideal
such that $(\calO,\pi)$ is a henselian pair. Let $G$ be a smooth,
separated $\calO$-group scheme and $\calG\to G$ the
dilatation of the trivial subgroup over $\calO/\pi$. If either
$\calO$ is local or~$G$ is affine, then the exact sequence of
Proposition ~\ref{lemma:short exact sequence}
induces an exact sequence of groups:
\[
1\longto \calG(\calO) \longto G(\calO) \longto G(\calO/\pi) \longto 1.
\]
\xcoro

 \part{Some applications }

\section{Models of group schemes, representation categories and Tannakian groups}\label{Tannak}

In several  mathematical theories, one finds the structure of a category with a {\it tensor product}, and one of the main goals of categorical Tannakian theory is to realize the latter categories  as representations of group schemes. If we deal with categories over a {\it field},  and this is a somewhat  well-known area  with \cite{deligne-milne82} being a fundamental reference, then dilatations have not played a role. In the case we deal with categories which are linear over a {\it discrete valuation} ring, a {\it Dedekind domain}, or  more complicated rings, the outputs are much  scarcer and the main reference is the beautiful, yet  arid, monograph \cite{saavedra72}. But in this situation,  dilatations have played a role. 

Following    \cite{dos_santos09}, N. D. Duong  and P. H. Hai \cite{duong-hai18} went into technical aspects of \cite{saavedra72} and produced  a more contemporaneous text to study tensor categories over Dedekind domain. This prompted further study        \cite{duong-hai-dos_santos18,hai-dos_santos21}; in these papers, the authors begin to look at N\'eron blowups (in the sense of Section \ref{neron}) and the resulting categories systematically. It is also useful to mention here the paper  \cite{csima-kottwitz10}, where the idea of looking at representation categories of  N\'eron blowups already appears.

  In this section we fix a discrete valuation ring $R$ with uniformizer $\pi$, residue field $k$ and fraction field $K$. We put $S= \Spec (R)$ and $S_i = \Spec ( R/(\pi ^{i+1}))$ for $i \in \bbN$.

\subsection{Group schemes from categories}\label{Tannak1}
 
Let $\mathcal T$ be a neutral Tannakian category over $R$ in the sense of \cite[Definition 1.2.5]{duong-hai18}. (We observe that in loc.cit., the authors require the {\it giving} of a faithful and exact tensor functor $\omega:\mathcal T\to R\text{-}\mathbf{mod}$ as part of the definition, but we shall only require its {\it existence}. The data of $\mathcal T$ {\it and} the functor $\omega:\mathcal T\to R\text{-}\mathbf{mod}$ should define a   neutral{\it ized} Tannakian category.)
The reader having encountered only (neutral) Tannakian categories over fields \cite[Section 2]{deligne-milne82}  should note that the distinctive property of $\mathcal T$ is a weakening of the existence of  ``duals'' \cite[Definition 1.7]{deligne-milne82}. This is    to be replaced by the  property that every object is a quotient of an ``object having a dual.''  That this property holds for representation categories of group schemes is   \cite[Proposition 3]{serre68}. (For a higher dimensional bases, see \cite[Lema 2.5]{thomason87}.)
But we face a  non-trivial requirement: for example, $\mathrm{Rep}_{W(\overline{\mathbb F}_p)}(\overline{\mathbb F}_p)$ fails to satisfy it \cite[Example 4.7]{hai-dos_santos21}. 

Once this definition of neutral Tannakian category is given, the main theorem of \cite[Theorem 2.11]{deligne-milne82} has his analogue in the present context: If $\omega:\mathcal T\to R\text{-}\mathbf{mod}$ is a faithful, $R$--linear and exact tensor functor (which exists by definition), then there exists an affine and {\it flat} group scheme $\Pi_{\mathcal T}$ over $R$ and an equivalence 
\[
\overline\omega: \mathcal T\longrightarrow\mathrm{Rep}_R(\Pi_{\mathcal T})
\]
such that composing $\overline\omega$ with the forgetful functor $\mathrm{Rep}_R(\Pi_{\mathcal T})\to R\text{-}\mathbf{mod}$ renders us $\omega$ back. See \cite[II.4.1.1]{saavedra72} and \cite[Theorem 1.2.2]{duong-hai18}.
It should be noted that the construction of $\Pi_{\mathcal T}$ {\it does depend on the functor $\omega$ chosen}. But we shall omit reference to $\omega$   in order to lighten notations.

Let us present some examples of categories to which the theory can be applied.

\begin{ex}Let $\Gamma$ be an abstract group and suppose that $R=k\llbracket \pi\rrbracket$. Then, the category of $R[\Gamma]$-modules which are of finite type over $R$ together with the forgetful functor produces  a neutral Tannakian category \cite[4.1]{hai-dos_santos21}. 
\end{ex}

\begin{ex}\label{05.06.2023--3jp}Let $X$ be a smooth and connected scheme over $R$,  $\mathcal{D}_{X/R}$ the ring of differential operators \cite[IV.16.8]{Gr60-67},  and  $\mathcal T^+$  the category of $\mathcal{D}_{X/R}$-modules which, as $\mathcal O_X$--modules, are coherent. We endow $\mathcal T^+$ with the usual tensor product of $\mathcal D_X$-modules  (in this level of generality, the tensor product is best grasped with the use of {\it stratifications}, cf. \cite[II.1.5.3, p. 105]{berthelot74}). 
 Using the fibre-by-fibre flatness criterion and \cite[proof of 2.16]{berthelot-ogus78}, one proves that an object $E\in \mathcal T^+$ is locally free if and only if it is $R$-flat.

Let now $\mathcal T$ be the full subcategory of $\mathcal T^+$    having
\[\left\{M\in\mathcal T^+\,:\,\begin{array}{c}\text{There exists $E\in\mathcal T^+$ which }
\\
\text{is $R$-flat and a surjection $E\to M$}
\end{array}\right\}\] 
as objects.   Once we give ourselves an $R$-point $x_0\in X(R)$, the functor  
\[
\mathcal T\longrightarrow R\text{-}\mathbf{mod},\qquad E\longmapsto \text{(global sections of) $x_0^*(E)$}
\] 
is exact, faithful and preserves tensor products, so that $\mathcal T$ is a neutral Tannakian category. For more details, see \cite{And01} and \cite{duong-hai-dos_santos18}. The flat affine  group scheme $\Pi_{\mathcal T}$, usually called the differential fundamental group scheme of $X$ at $x_0$ and  frequently  denoted  by $\Pi^{\mathrm{diff}}(X,x_0)$,   is difficult to be described in general. 
\end{ex}

\begin{ex}We assume that $R$ is Henselian and Japanese, e.g.  $R$ is     complete.  Let $X$ be an irreducible, proper and flat $R$-scheme with geometrically reduced fibres. Let $x_0\in X(R)$. 
Given a coherent sheaf $E$ on $X$, we say that $E$ is {\it trivialized by a proper morphism} if there exists a surjective and proper morphism $\psi:Y\to X$ such that $\psi^*E$ ``comes from $S=\mathrm{Spec}\,R$'', by which we mean that $\psi^*E$ is the pull-back of a module via the structural morphism $Y\to S$. Let $\mathcal T^+$ be the full subcategory of the category of coherent modules on $X$ having as objects those sheaves which are trivialized by a proper morphism. Proceeding along the lines of Example \ref{05.06.2023--3jp}, it is possible to construct a smaller full subcategory $\mathcal T$ of $\mathcal T^+$  such that, endowing $\mathcal T$ with the tensor product of sheaves, the functor 
\[
\mathcal T\longrightarrow R\text{-}\mathbf{mod},\qquad E\longmapsto \text{(global sections of) $x_0^*(E)$}
\]
produces a neutral Tannakian category and with it a flat group scheme   $\Pi^{\mathrm N}(X,x_0)$. Details are in \cite{hai-dos_santos23}. This is the analogue theory of Nori's theory for the fundamental group scheme \cite{nori76} in the relative setting, and one objective is to show that $\Pi^{\mathrm N}(X,x_0)$ is {\it pro-finite}. See \cite[Theorem 8.8]{hai-dos_santos23}.   
\end{ex}
 
\subsection{Galois-Tannaka group  schemes}\label{Tannak2}
One obvious strategy to study Tannakian categories is to filter them by categories ``generated'' by a single object, just as in studying Galois groups it is fundamental to study finite extensions. Let $\omega:\mathcal T\to R\text{-}\mathbf{mod}$ be    as in the previous section so that $\mathcal T$ is equivalent to $\mathrm{Rep}_R(\Pi)$ for some affine and flat group scheme $\Pi$. We shall take this equivalence as an equality, but we warn the reader that the structure of $\Pi$ should be considered as being very complicated (just as is that of an absolute Galois group). 

\defi
Let $M\in\mathcal T$ be an object possessing a dual $M^\vee$ and for each couple of non-negative integers $a,b$,  define $\mathbf T^{a,b} M$ as $M^{\otimes a}\otimes M^{\vee\,\otimes b}$. 
Then,   $\langle M\rangle_\otimes$ is  the full subcategory of $\mathcal T$ having as objects those which are quotients of subobjects of elements of the form 
\[
\mathbf T^{a_1,b_1}M\oplus\cdots\oplus \mathbf T^{a_r,b_r}M, 
\]
for varying $r$, $a_1,\ldots,a_r$, $b_1,\ldots,b_r$.  
The Tannakian group scheme associated to $\langle M\rangle_\otimes$ via $\omega$ will be called here the {\it (full) Galois-Tannaka group (scheme)} of $M$. 
\xdefi

As we concentrate on a neutral Tannakian category, it is instructive to note that the splicing of $\mathcal T$ by various $\langle M\rangle_\otimes$ amounts to looking at various ``images'' of $\Pi$.  Before entering this topic, recall that, given a base field  $F$ and   a morphism $\varphi:G'\to G$ of affine group schemes over $F$,  the closed image $\mathrm{Im}_\varphi$   \cite[I.9.5]{Gr60-67} is a {\it closed subgroup  scheme} of $G$  such that the natural morphism  $G'\to \mathrm{Im}_\varphi$ is {\it faithfully flat} \cite[Theorem on 15.1]{waterhouse79}. In this case, $\mathrm{Im}_\varphi$ enjoys both ``desirable properties'' of 
an image. 

\begin{dfn}Let   $\rho:\Pi\to G$ be a morphism of flat and affine group schemes over $R$. Define the {\it restricted} image of $\rho$, denoted $\mathrm{Im}_\rho$, as the affine scheme associated to the algebra  
\[
\text{$B_\rho$ = Image of $\mathcal O(G)\to \mathcal O(\Pi)$.}
\] (In other words, $\mathrm{Im}_\rho$ is the ``closed'' image of $\rho$ \cite[I.9.5]{Gr60-67}.)
Define its {\it full} image $\mathrm{Im}'_\rho$ as being the affine scheme associated to   
\[\tag{\S}
B_\rho'=\{f\in  \mathcal O(\Pi)\,:\,\text{   $\pi^mf\in B_\rho$, for some $m\ge0$}\}.
\]
\end{dfn}
It is not difficult to see that $\mathrm{Im}_\rho$ and $\mathrm{Im}'_\rho$ are affine group schemes. 
With these definitions,         $\rho$ factors as
\[\tag{$\dagger$}
\Pi \stackrel \psi\longrightarrow  \mathrm{Im}'_\rho \stackrel u \longrightarrow   \mathrm{Im}_\rho \stackrel \iota\longrightarrow    G,
\]
where   $\iota$ is a {\it closed immersion} and $u$ induces an isomorphism between generic fibres. A fundamental result \cite[Theorem 4.1.1]{duong-hai18} now assures that $\psi$ is faithfully flat, so that the terms ``images''  are justified and the factorization in $(\dagger)$  is called the {\it diptych} of $\rho$.

 In addition, if \[\rho_K:\Pi\otimes K\longrightarrow G\otimes K\] stands for the morphism obtained from $\rho$ by base-change to $K$, then we have 
\[
\mathrm{Im}'_\rho\otimes K=\mathrm{Im}_\rho\otimes K=\mathrm{Im}(\rho_K).
\]

\begin{propo}[{\cite[Proposition 4.10]{duong-hai-dos_santos18}}]\label{05.06.2023--1jp}
Let $M$ be a finite and free $R$-module affording a representation of $\Pi$ and let $\rho:\Pi\to \mathrm{GL}(M)$ be the associated  homomorphism. Then  the obvious functor $\mathrm{Rep}_R(\mathrm{Im}'_\rho)\to\mathrm{Rep}_R(\Pi)$ defines an equivalence between $\mathrm{Rep}_R(\mathrm{Im}'_\rho)$ and $\langle M\rangle_\otimes$. Put differently, $\mathrm{Im}'_\rho$ is the Galois-Tannaka group of $M$ (in $\mathrm{Rep}_R(\Pi)$). 
\end{propo}

\rema Let $\mathrm{Rep}_R^\circ(\mathrm{Im}_\rho)$ be the full subcategory of $\mathrm{Rep}_R(\mathrm{Im}_\rho)$ consisting of objects having a dual; it is possible to show that $\mathrm{Rep}_R^\circ(\mathrm{Im}_\rho)$ is equivalent to a full subcategory, $\langle M\rangle_\otimes^s$, of $\langle M\rangle_\otimes$. See Proposition 4.10 in \cite{duong-hai-dos_santos18}. 
On   the other hand, the functor  $\mathrm{Rep}_R(\mathrm{Im}_\rho)\to \mathrm{Rep}_R(\Pi)$ may easily fail to be full.  
\xrema

{\it From now on, we give ourselves a representation $\rho:\Pi\to \mathrm{GL}(M)$} as in Proposition \ref{05.06.2023--1jp}. 
It is at this point that  the theory over $R$ begins to distance itself from the theory over a field in a significant way. Indeed, in the case of a base-field,  Galois-Tannaka group schemes  are known to be of finite type \cite[Proposition 2.20]{deligne-milne82}. {\it This is not unconditionally true over $R$} since in order to construct $\mathrm{Im}'_\rho$, it was required to ``saturate'' the ring $B_\rho$ in (\S). On the other hand,  the morphism  $\mathrm{Im}_\rho\to \mathrm{GL}(M)$ is a closed immersion and $\mathrm{Im}_\rho$ is of finite type.

\defi A {\it model} of a   group scheme of finite type $G$ over $K$ is a flat group scheme $\bbG$ over $R$ such that $\bbG \otimes_R K \cong  G$,  as $K$-group schemes. We often   identify
 $G$ and the generic fibre $\bbG\otimes_RK$. 
A morphism of models $\bbG \to \bbG'$ of 
$G$ is a morphism $ \bbG \to \bbG'$ of group schemes over 
$R $ which induces the identity on $G$ once unravelled the proper identifications. 
\xdefi

\rema The definition of   model used here    differs from the one used in \cite[\href{https://stacks.math.columbia.edu/tag/0C2R}{Tag 0C2R}]{stacks-project} and \cite{WW80}  namely, we do not assume our models to be of finite type over $R$. 
\xrema

With this terminology, $\mathrm{Im}_\rho$  and $\mathrm{Im}'_\rho$ are models of        $\mathrm{Im}(\rho_K)$. 
 A well-known result of Waterhouse-Weisfeler about the relations between models is the following.

\begin{thm}[{\cite[Theorem 1.4]{WW80}, \cite[Theorem 2.11]{duong-hai-dos_santos18}}]\label{thWW} Let $v:G' \to G $ be a morphism of flat $S$-group schemes      such that $v$ is an isomorphism  on generic fibres. Then $v$ is a   composite of mono-centered N\'eron blowups (along the divisor defined by $\pi$). In other words, a morphism of models of finite type is a composite of mono-centered N\'eron blowups. If $G$ and $G'$ are of finite type, then the number of N\'eron blowups is finite. 

More precisely: 
Define $v_0=v$ and $G_0=G$, and assume that $v_0,\ldots,v_n$ have been inductively constructed. Then,   put     
\[
G_{n+1}=\mathrm{Bl}_{\mathrm{Im}_{v_n\otimes k}}^{G_n\otimes k}(G_n).
\](Recall that $k$ is the residue field.)
Once this is done, define   $v_{n+1}:G'\to G_{n+1}$ to  be the morphism obtained from the universal property of $\mathrm{Bl}_{\mathrm{Im}_{v_n\otimes k}}^{G_n\otimes k}(G_n)$ (cf. Proposition \ref{Neron.blow.lemm}). Then 
\[
\varprojlim_nv_n:G'\longrightarrow \varprojlim_nG_n
\]
is an isomorphism. In particular, if for some $n\in\mathbb N$ the homomorphism  $v_n\otimes k$ is faithfully flat, then $G'\simeq G_n$.  
\end{thm}

As was mentioned before, it is   possible that $\mathrm{Im}'_\rho$ fails to be of finite type and hence the number of N\'eron blowups proposed by Theorem \ref{thWW} to   describe  $u:\mathrm{Im}_\rho'\to\mathrm{Im}_\rho$
may  be infinite. But in some cases, it does happen that  the number of N\'eron blowups is finite and a condition for this situation is described in  
 Theorem \ref{thWW}.
At this point, we remind the reader that in the situations  we have in mind, the group scheme $\Pi$ is usually extremely complicated and the determination of the image of a morphism $\Pi\otimes k\to G\otimes k$, so that it is possible to apply the last claim in Theorem \ref{thWW}, can only be achieved on the side of $\mathrm{Rep}_R(\Pi)$. 

It then becomes relevant to determine faithful representations of N\'eron blowups. (Here, we say that a representation is {\it faithful} if the morphism to the associated general linear group is a closed immersion. This is not universally adopted.) The next result explains how to proceed in certain cases.

\begin{thm}[{\cite[Corollary 3.6]{duong-hai-dos_santos18}}]Let $G$ be an affine and flat group scheme of finite type over $S$. Let $M$ be a finite and free $R$-module affording a faithful representation of $G$. Given $m\in M$, let 
\[
H_0 =\text{stabilizer of $m\otimes1\in M\otimes k$}
\] 
in $G\otimes k$. Let $G'=\mathrm{Bl}_{H_0}^{G\otimes k}(G)$.  Then, letting $\mathbbm1=R$ stand for the trivial representation of $G'$, the obvious map $\mathbbm1\to M\otimes k$ determined by $1\mapsto v\otimes 1$
is   $G'$-equivariant and the fibered product 
\[\tag{$\P$}
M':=M\underset {M\otimes k}{\times}\mathbbm1
\]
now affords a faithful  representation of $G'$.
\end{thm}

Let us illustrate the above result with a simple example showing how to compute a Galois-Tannaka group. 

\begin{ex}Let $k$ be of characteristic zero and  $\mathcal T$ be the category of representations of the abstract group $\mathbb Z$ on finite $R$-modules. It is not difficult to see that $\mathcal T$ is neutral Tannakian \cite[Corollary 4.5]{hai-dos_santos21}. Let $\mathbb Z$ act on  
$M=R$ by $\gamma\cdot r=(1+\pi)^\gamma r$ and write $\rho:\Pi\to   \mathrm{GL}(M)(\simeq\mathbb G_{m,R})$  for the associated morphism of group schemes. It is not difficult to see that $\mathrm{Im}_\rho=\mathbb G_{m,R}$ and we wish to compute $\mathrm{Im}_\rho'$. As mentioned above, the construction ($\S$) is of little use since $\mathcal O(\Pi)$ is   mysterious. On the other hand, we know that $\Pi$ will act trivially on   $M\otimes k$ because $\mathbb Z$ does. We then need to perform the ``dilatation'' $M'$ of $M$ as in $(\P)$, which is a faithful representation of the N\'eron blowup $\mathrm{Bl}_{\{e\}}^{\mathbb G_m\otimes k}(\mathbb G_m)=:G'$. Before proceeding, we offer a down-to-earth description of 
$G'$ hoping that it will be useful in   the calculations to come. 

By definition,   $\mathbb G_{m,R}=\mathrm{Spec}\,R[x,x^{-1}]$ and comultiplication on the Hopf algebra $R[x,x^{-1}]$
is given by $x\mapsto x\otimes x$. 
Introducing the function  $q=x-1$, comultiplication becomes $q\mapsto 1\otimes q  +q\otimes 1+q\otimes q$. Let   $q' =\pi^{-1}q\in \mathcal O(G')$. Then, comultiplication on  $\mathcal O(G')$ is defined by $q'\mapsto 1\otimes q'  +q'\otimes 1+\pi q'\otimes q'$. In particular, comultiplication on $\mathcal O(G')\otimes k$ sends  $q'$ to $q'\otimes 1+1\otimes q'$. From this the reader should be able to see that $G'\otimes k\simeq \mathbb G_{a,k}$.

We now return to the side of representations.
Since $M\otimes k$ has the {\it trivial} action of $\mathbb Z$, it   follows that $\Pi$ acts trivially on it as well. That is, $M\otimes k$ is the trivial representation of $\Pi\otimes k$. Therefore, $\rho$ factors through $G'\to\mathbb G_{m,R}=\mathrm{GL}(M)$. 
Let $\rho':\Pi\to G'$ be the resulting morphism. We claim that $\rho'$ is faithfully flat, and for that it is enough to show that $\rho'$ is faithfully flat on both the generic and special fibres \cite[Theorem 4.1.1]{duong-hai18}. That this is the case for the generic fibre is already assured and   we are left with the analysis of the special fibre.  
Now, $G'\otimes k\simeq\mathbb G_{a,k}$  and any  element  of $k\setminus\{0\}$  generates a dense subgroup of $\mathbb G_{a,k}$. Hence, if we show that $\rho'$ does not induce the {\it trivial} morphism on $k$-rational points, we then guarantee that $\rho'\otimes k$ has a dense image, and therefore is faithfully flat. We must hence consider the faithful representation $M'$ of $G'$. 
 
The elements  $m_1:=(\pi,0)$ and $m_2:=(1,1)$ obviously form a basis for $M'$ and hence the resulting representation of $\mathbb Z$ is defined by 
\[
\gamma\longmapsto\begin{pmatrix}1+\pi&1\\0&1\end{pmatrix}^\gamma.
\] 
Consequently, $M'\otimes k$ is {\it not} the trivial representation of $\mathbb Z$ and therefore it is  {\it not either}  the trivial representation of $\Pi\otimes k$. Hence, $\rho'\otimes k$ is not the trivial morphism and we conclude that $\rho'\otimes k$ is faithfully flat. Our verification that 
$\mathrm{Im}_{\rho}'$ is $G'$ is finished. 
\end{ex}

On the other hand, when the number of N\'eron blowups envisaged by Theorem \ref{thWW} is infinite, 
a general principle behind \cite{duong-hai-dos_santos18, hai-dos_santos21} is that the  Galois-Tannaka groups can be obtained from group schemes of finite type  via certain special types of  (what we now call) {\it multi-centered}  N\'eron blowups.  This is treated in the next section. 

\subsection{  N\'eron blowups of formal subgroup schemes}\label{Tannak3}
Multi-centered dilatations having divisors which are supported on the same space have been studied more closely. For an affine  group scheme $G$ over $R$, we shall write $\widehat G$ for the completion $G_{/G_0}$ of $G$ along its closed fiber \cite[I.10]{Gr60-67}.

\begin{dfn}[{\cite[Definition 5.6]{duong-hai-dos_santos18}}]\label{15.05.2023--1} Let  $G\to S$ be an affine flat group scheme of finite type. For each $i\in\mathbb N$, let $G_i$ be the $S_i$-group scheme $G\times_{S}S_i$, and let  $H_i\to S_i$ be a closed,  $S_i$--flat, subgroup-scheme of  $G_i$. Assume, in addition, that the natural base-change morphism 
\[
H_{i+1}\times_{S_{i+1}} S_i\longrightarrow G_{i+1}\times_{S_{i+1}} S_i=G_i
\]  
defines an isomorphism  $H_{i+1}\times_{S_{i+1}} S_i\simeq H_i$ of group schemes. 
Said differently, the family $\{H_i\}$ induces a {\it formal closed subgroup scheme $\mathfrak H$ of $\widehat G$}.
We define the {\it     N\'eron blowup} of $G$ along $\mathfrak H$, call it $\mathrm{Bl}^{\widehat G}_{\mathfrak H} G$, as being $\mathrm{Bl}_{\{H_i\}}^{\{G_
i\}}G\to G$. 
\end{dfn}

\begin{rmk} If the formal scheme  $\mathfrak H$ is ``algebraizable'', meaning that it comes from a closed and flat subgroup scheme $H\subset G$, this is mentioned in \cite[\S7.2]{PY06}. 
\end{rmk}

\begin{ex}\label{05.06.2023--2jp}Let $p$ be a prime number,  $R=\mathbb Z_p$ and $G=\mathbb G_{a,R}$. It then follows that the completion of $G$ along    its closed fibre (i.e. the divisor defined by $(p)$) is $\mathrm{Spf}\,\mathbb Z_p\langle x\rangle$, where $\mathbb Z_p\langle x\rangle$ is the subring of $\mathbb Z_p\llbracket x\rrbracket$ consisting of power series $\sum_na_nx^n$ such that $\lim a_n=0$. 
Let $\mathfrak H$ be the closed formal subscheme of $\widehat G$ determined by the ideal $(x)\subset \mathbb Z_p\langle x\rangle$.  Then, it is not difficult to see that $\mathrm{Bl}_{\mathfrak H}^{\widehat G}G$ is the group scheme determined by the Hopf subalgbra $A=\{P\in \mathbb Q_p[x]\,:\,P(0)\in \mathbb Z_p\}$. Note that $\mathrm{Bl}_{\mathfrak H}^{\widehat G}G\otimes \mathbb F_p$ is the trivial group scheme, while $\mathrm{Bl}_{\mathfrak H}^{\widehat G}G\otimes \mathbb Q_p$ is $\mathbb G_{a,\mathbb Q_p}$. In particular, the dimension of the generic and closed fibres is distinct, even though $\mathrm{Bl}_{\mathfrak H}^{\widehat G}G$ is itself flat over $\mathbb Z_p$. Note, on the other hand, that the $\mathbb Z_p$-module $A$ contains a copy of $\mathbb Q_p$ and hence fails to be {\it projective} over $\mathbb Z_p$. This seemingly harmless property is the cause of complications in the category of representations \cite[Proposition 6.19]{hai-dos_santos21} as the {\it inexistence} of  intersections of subrepresentations.  
\end{ex}

\begin{ex}[{\cite[4.3]{hai-dos_santos21}}]Let   $R=k\llbracket \pi\rrbracket$, where $k$ is a field of characteristic zero and let $G=\mathbb G_{a,R}\times_R\mathbb G_{m,R}$. Letting $x$ stand for  ``the'' coordinate of $\mathbb G_{a,R}$ and $y$ for  ``the'' coordinate of $\mathbb G_{m,R}$, we define  
\[
e^{\pi x}=\sum_{i=0}^\infty \frac {\pi^i  }{i!}x^i;
\]  
this is an element of $\widehat{\mathcal O(G)}$. It is not difficult to see that 
$y-e^{\pi x}$ cuts out a closed and formal subgroup  scheme of $\widehat G$, call it $\mathfrak H$, and hence we obtain a model $\mathrm{Bl}_{\mathfrak H}^{\widehat G}\to G$. Note that $\mathfrak H$ is not algebraizable.  Differently from the situation in Example \ref{05.06.2023--2jp}, the $R$-module  $\mathcal O(\mathrm{Bl}_{\mathfrak H}^{\widehat G})$ is \emph{projective}.
\end{ex}

One important consequence of the procedure of taking formal blowups is the following. It says that, in some contexts, all the information concerning a model of a group scheme can be encoded in a formal N\'eron blowup (Theorem \ref{thm3.3hds21}). 

\begin{thm}[{\cite[Corollary 3.3]{hai-dos_santos21}}] \label{thm3.3hds21} Suppose that the $R$ is     complete and of residual characteristic zero.  Let  $\mathcal G\to G$ be a morphism of affine and flat  $R$-group schemes  inducing an isomorphism on the  generic fibres, and suppose in addition that $G$ is of finite type. Then, there exists a group scheme $G'$ over $R$, flat and of finite type, and a morphism of group schemes $G'\to G$ which is an isomorphism on generic fibres, a closed and formal subgroup scheme $\mathfrak H'$ of $\widehat G'$, and an isomorphism 
\[
\mathcal G\stackrel \sim\longrightarrow \mathrm{Bl}^{\widehat  G'}_{\mathfrak H'} G'.
\]
\end{thm}

\rema Under the assumptions of Theorem \ref{thm3.3hds21},  Theorems \ref{thWW} and \ref{thm3.3hds21} together say that any morphism of models $G' \to G $ with $G$ of finite type over $R$ is obtained as a composite of multi-centered N\'eron blowups, and more precisely as a formal N\'eron blowup composed by several mono-centered N\'eron blowups. \xrema

\section{Congruent isomorphisms and relations with Bruhat-Tits buildings, the Moy-Prasad isomorphism and admissible representations of $p$-adic groups}
\label{RepIso}

In this section we report on congruent isomorphisms. Let $(\calO,\pi)$ be a henselian pair where $\pi\subset \calO$ is
an invertible ideal. 
Let us start with the following result proved in \cite{MRR20}.

\theo[\cite{MRR20}] (Congruent isomorphism) \label{moncong}
Let $r,s$ be integers such that $0\le r/2\le s\le r$.
 Let $G$ be a smooth, separated
$\calO$-group scheme. Let $G_r$ be the $r$-th iterated dilatation of
the unit section (i.e. $G_r = \Bl _{e_G}^{\calO /\pi^r} G$) and $\frakg_r$ be its Lie algebra. If $\calO$ is
local or $G$ is affine, there is a canonical and functorial
isomorphism of groups:
\[\tag{$\star$} G_s(\calO)/G_r(\calO) \isomto \frakg_s(\calO)/\frakg_r(\calO).
\]
\xtheo

\rema \label{rem1iso} We comment on works prior to Theorem \ref{moncong}. \begin{enumerate} \item 
In the case of an affine and smooth group scheme
over a discrete valuation ring, the isomorphism of
Theorem~\ref{moncong} appears without proof in \cite[proof of Lemma 2.4]{Yu15}. 

\item 
The proof of Theorem \ref{moncong} relies on Proposition \ref{exceptional divisor} and Theorem \ref{theo:dilatations of subgroups} whose proofs (given in \cite[Prop. 2.9]{MRR20} and \cite[Th. 3.5]{MRR20}) basically consist in playing and computing with quasi-coherent ideals. These computations on quasi-coherent ideals in \cite{MRR20} were partly motivated by related computations on ideals done in the affine case in \cite[Appendix A]{Ma19t} to understand the congruent isomorphism. The statement of \cite[Th. 3.5]{MRR20} is moreover partly inspired by \cite[Th. 1.5, Th. 1.7]{WW80}. 

\item  
 If $G=\mathbb G_m/\bbZ_p$,  the isomorphism ($\star$) follows from the multiplicative structure of $\bbZ _p $ cf. e.g. \cite{He1913}, \cite{Ha50} and \cite[Chap. 15]{Ha80}. Similar isomorphisms for matrix groups over non-Archimedean local fields were used in \cite[p.~442 line 1]{Ho77}, \cite[2.13]{Mo91}, \cite[p.~22]{BK93}, \cite[p.~337]{Sec04} and many other references to study admissible representations of $p$-adic classical groups. In the matrix case, the filtrations involved are defined using matrix theoretic descriptions and avoiding scheme theoretic tools. For general reductive groups over non-Archimedean local fields, such kind of isomorphisms were introduced and used in  \cite[\S2]{PR84}, \cite[\S2]{MP94}, \cite{MP96},
\cite[\S1]{Ad98}, \cite[\S1]{Yu01} to study admissible representations. In the reductive case, the filtrations involved are the Moy-Prasad filtrations  \cite{MP94}, \cite{MP96} and the isomorphism is called the Moy-Prasad isomorphism. These filtrations are defined for points in the Bruhat-Tits building using the associated valued root datum \cite{BT72} \cite{BT84}. The Moy-Prasad isomorphim in these references was defined using somehow ad hoc formulas and the valued root datum, in particular avoiding the congruent isomorphism. However it is known that one has to modify the original Moy-Prasad filtrations to ensure the validity of the Moy-Prasad isomorphism in full generality, cf. \cite[§0.3]{Yu15} and \cite[§13]{KP22}.

 \end{enumerate}
\xrema 

If $G =  \mathrm{GL}_2 /\bbZ_p $, then $G_n (\bbZ_p) =\begin{pmatrix}1+\frakp^n & \frakp^n\\ \frakp^n & 1+ \frakp^n    \end{pmatrix} \subset \mathrm{GL}_2 (\bbZ _p) $ and $\mathfrak{g} _n (\bbZ_p)=\begin{pmatrix}\frakp^n & \frakp^n \\ \frakp^n & \frakp^n  \end{pmatrix} \subset M_2 (\bbZ_p)$ for any $n>0$. The isomorphism $(\star)$ gives us, for pairs $(r,s)$ such that $0 < \frac{r}{2} \leq s \leq r $, isomorphisms
\[ \tag{$ * $} \begin{pmatrix}1+\frakp^s & \frakp^s\\ \frakp^s & 1+ \frakp^s    \end{pmatrix} /  \begin{pmatrix}1+\frakp^r & \frakp^r \\ \frakp^r & 1+ \frakp^r \end{pmatrix} \cong \begin{pmatrix}\frakp^s & \frakp^s \\ \frakp^s &  \frakp^s  \end{pmatrix} /  \begin{pmatrix}\frakp^r & \frakp^r \\ \frakp^r &  \frakp^r \end{pmatrix}.\]  These maps are given by $[1+M] \mapsto [M]$. Using the formula $ [1+M] \mapsto [M] $, it is elementary to check that we have other isomorphisms of abstract groups
 \begin{align*} \tag{$ ** $} \begin{pmatrix}1+\frakp^3 & \frakp^3\\ \frakp^3 & 1+ \frakp^3 \end{pmatrix} /  \begin{pmatrix}1+\frakp^5 & \frakp^6 \\ \frakp^6 & 1+ \frakp^5 \end{pmatrix} \cong  \begin{pmatrix}\frakp^3 & \frakp^3 \\ \frakp^3 &  \frakp^3  \end{pmatrix} /  \begin{pmatrix}\frakp^5 & \frakp^6 \\ \frakp^6 &  \frakp^5 \end{pmatrix},\\
 \tag{$***$}\begin{pmatrix}1+\frakp^3 & \frakp^9\\ \frakp^3 & 1+ \frakp^3 \end{pmatrix} /  \begin{pmatrix}1+\frakp^6 & \frakp^9 \\ \frakp^6 & 1+ \frakp^6 \end{pmatrix} \cong  \begin{pmatrix}\frakp^3 & \frakp^9 \\ \frakp^3 &  \frakp^3  \end{pmatrix} /  \begin{pmatrix}\frakp^6 & \frakp^9 \\ \frakp^6 &  \frakp^6 \end{pmatrix}.
 \end{align*}
These isomorphisms are obtained as follows from the point of view of dilatations.

\theo[\cite{Ma23d}] \label{isocorocongru} (Multi-centered congruent isomorphism)
Let $G$ be a separated and smooth group scheme over $S$. Let $H_0 \subset H_1 \subset \ldots \subset  H_k$ be closed subgroup schemes of $G$ such that $H_i$ is smooth over $S$ for $0 \leq i \leq d$ and $H_0=e_G$. Let $s_0 , s_1 , \ldots , s_k$ and $r_0 , r_1 , \ldots , r_k$ be in $\bbN$ such that  
 \begin{enumerate}
 \item $s_i \geq s_0 $ and $r_i \geq r_0 $ for all $i \in \{0, \ldots , k\}$,
 \item $ r_i \geq s_i $ and $r_i-s_i \leq s_0$ for all $i \in \{0, \ldots , k \}$.
  \end{enumerate}  Assume that $G$ is affine or $\calO $ is local.  Then we have a canonical isomorphism of groups
 \[ \Bl_{H_0, H_1 , \ldots , H_k}^{s_0 ,~ s_1 ,~ \ldots , s_k} G (\calO)/  \Bl_{H_0, H_1 , \ldots , H_k}^{r_0 ,~ r_1 ,~ \ldots , r_k} G  (\calO)  \cong \Lie (\Bl_{H_0, H_1 , \ldots , H_k}^{s_0 ,~ s_1 ,~ \ldots , s_k} G )(\calO) / \Lie ( \Bl_{H_0, H_1 , \ldots , H_k}^{r_0 ,~ r_1 ,~ \ldots , r_k} G   ) (\calO) \]
where $\Bl_{H_0, \ldots , H_k}^{t_0 ,~ \ldots , t_k} G  $ denotes $ \Bl_{H_0,  ~\ldots~ , H_k}^{\calO/\pi ^{t_0}, \ldots ,  \calO /\pi^{t_k}} G $ for any $t_0 , \ldots , t_k \in \bbN$.
\xtheo

Now let $G$ be $ GL_2 /\bbZ _p$. Let $ e_G \subset G$ be the trivial subgroup. Let $T$ be the diagonal split torus in $G$. Let $B$ be the lower triangular Borel in $G$ over $\bbZ_p$. \begin{enumerate} \item The isomorphism $(**)$ above is given by Theorem \ref{isocorocongru} with 
$(\calO , \pi ) =(\bbZ _p,  \mathfrak{p} )$, $H_0=e_G$, $H_1= T$, $s_0 =3, s_1=3, r_0= 5 $ and $r_1=6$.
\item The isomorphism $(***)$ above is given by Theorem \ref{isocorocongru} with 
$(\calO , \pi )= (\bbZ _p,  \mathfrak{p} )$, $H_0=e_G$, $H_1= B$, $s_0 =3, s_1=9, r_0= 6 $ and $r_1=9$.
\end{enumerate}

\rema We comment on Theorem \ref{isocorocongru}.
\begin{enumerate} 
\item Theorem \ref{isocorocongru} corresponds to \cite[Corollary 8.3]{Ma23d}, a slightly more general result is given by \cite[Theorem 8.1]{Ma23d}.
\item The proof of Theorem \ref{isocorocongru} given in \cite{Ma23d} relies on Theorem \ref{moncong} and the study of multi-centered dilatations.

\item  Note that \cite[Lemma 1.3]{Yu01} provides a comparable "multi-centered" isomorphism, in the framework of reductive groups over non-Archimedean local field.
 \end{enumerate} \xrema

 Recall that dilatations of schemes over discrete valuation rings are used in Yu's approach \cite{Yu15} on Bruhat-Tits theory for reductive groups over henselian discrete valuation field with perfect residue field. We refer to the monograph by Kaletha and Prasad \cite{KP22} that include among other things a detailled exposition of \cite{Yu15}. 
 The congruent isomorphism (Theorem \ref{moncong}) and its proof (relying on several results of \cite{MRR20}) are now used as foundation to prove the Moy-Prasad isomorphism for reductive groups mentioned in Remark \ref{rem1iso}, cf. \cite[Theorem 13.5.1 and its proof, Proposition A.5.19 (3) and its proof]{KP22}. As a consequence, dilatations and congruent isomorphisms are now part of the foundation to study admissible representations of reductive $p$-adic groups. Furthermore, other connections between dilatations and groups used to study admissible representations can be found in \cite[§10]{Yu15} and \cite[Example 1.4]{Ma23d}. Reciprocally, the problem of constructing supercuspidal representations of $p$-adic groups (cf. e.g. \cite[Remark/Conclusion]{Ma19o}, \cite{Ma19t, MY22}), or more generally types in the sense of \cite{BK98}, could continue to be a source of inspiration to expend the theory of dilatations.
 
  \rema 
 As we explained before, the book \cite{KP22} provides a carefully written new approach to Bruhat-Tits theory in the case of discrete valuations. This beautiful monograph uses the theory of dilatations to deal with integral models whereas the original Bruhat-Tits theory \cite{BT84} did not. Let us quote \cite[Introduction]{KP22}:
\begin{center}
   ``\textit{Next we
turn to the construction of integral models [...]. Instead of using the approach of Bruhat–Tits via schematic
root data, we employ a simpler and more direct method due to Jiu-Kang Yu
\cite{Yu15}, based on the systematic use of N\'eron dilatations.}''
 \end{center} 
The book \cite{KP22} offers an appendix on dilatations. Though \cite[Appendix A.5]{KP22} takes into account the treatment of dilatations in \cite{MRR20}, it restricts to the framework of discrete valuations.
Originally, Bruhat-Tits theory \cite{BT84} deals also with non discrete valuations, it is natural to ask whether the modern and general approach to dilatations of schemes initiated in \cite{MRR20} could help to provide a more conceptual treatment (in the spirit of \cite{Yu15} and \cite{KP22}) of some parts of \cite{BT84}. Bruhat-Tits theory and dilatations over non discrete valuations were used in \cite{RTW10} and \cite{Ma22} to study Berkovich's point of view \cite[Chap. 5]{Be90} on Bruhat-Tits buildings of reductive groups over discrete and non-discrete valuations (cf. e.g. \cite[1.3.4]{RTW10} for precise assumptions).
 \xrema

\section{Torsors, level structures and shtukas} \label{Torsors.blowup.sec}

\label{TorShtuk}

In this subsection, we explain that many level structures
on moduli stacks of $G$-bundles are encoded in torsors under Néron blowups of $G$ following \cite{MRR20}.
Assume that $X$ is a smooth, projective, geometrically irreducible curve
over a field $k$ with a Cartier divisor $N\subset X$, that $G\to X$ is a
smooth, affine group scheme and that $H\to N$ is a smooth closed subgroup
scheme of $G|_N$. In this case, the N\'eron blowup $\calG \to X$ $(\calG = \Bl_{H}^{N}G) $ is a smooth,
affine group scheme. Let $\Bun_G$ (resp.~$\Bun_\calG$) denote the moduli
stack of $G$-torsors (resp.~$\calG$-torsors) on $X$.  This is a quasi-separated,
smooth algebraic stack locally of finite type over $k$
(cf.~e.g.~\cite[Prop.~1]{He10} or \cite[Thm.~2.5]{AH19}).
Pushforward of torsors along $\calG\to G$ induces a morphism
$\Bun_\calG\to \Bun_G$, $\calE\mapsto \calE\x^\calG G$.
We also consider the stack $\Bun_{(G,H,N)}$ of $G$-torsors on $X$ with level-($H$,$N$)-structures, cf. \cite[Definition 4.5]{MRR20}.
Its $k$-points parametrize pairs $(\calE,\be)$ consisting of a $G$-torsor $\calE\to X$ and a section $\be$ of the fppf quotient $(\calE|_{N}/H)\to N$, i.e., $\be$ is a reduction of $\calE|_N$ to an $H$-torsor.

\prop[\cite{MRR20}]  \label{torsorprop}
 There is an equivalence of $k$-stacks
\[
\Bun_\calG\;\overset{\cong}{\longto}\;\Bun_{(G,H,N)},\;\; \calE\longmapsto (\calE\x^\calG G,\, \be_{\on{can}}),
\]
where $\be_{\on{can}}$ denotes the canonical reduction induced from the factorization $\calG|_N\to H\subset G|_N$. 
\xprop

Thus, many level structures are encoded in torsors under N\'eron blowups.
This construction is also compatible with the adelic viewpoint as follows. Let $|X|\subset X$ be the set of closed points, and let $\eta \in X$ be the generic point.
We denote by $F=\kappa(\eta)$ the function field of $X$.
For each $x\in |X|$, we let $\calO_x$ be the completed local ring at $x$ with fraction field $F_x$ and residue field $\kappa(x)=\calO_x/\frakm_x$. 
Let $\bbA:=\bigsqcap'_{x\in |X|}F_x$ be the ring of adeles with subring of integral elements $\bbO=\bigsqcap_{x\in |X|}\calO_x$.

\prop
Assume either that $k$ is a finite field and $G\to X$ has connected fibers, or that $k$ is a separably closed field. The N\'eron blowup $\calG\to X$ is smooth, affine with connected fibers, and there is a commutative diagram of groupoids
\[
\begin{tikzpicture}[baseline=(current  bounding  box.center)]
\matrix(a)[matrix of math nodes, 
row sep=1.5em, column sep=2em, 
text height=1.5ex, text depth=0.45ex] 
{ 
\Bun_\calG(k)& \bigsqcup_\ga G_\ga(F)\big\bsl\big(G_\ga(\bbA)/\calG(\bbO )\big) \\ 
\Bun_G(k)& \bigsqcup_\ga G_\ga(F)\big\bsl\big(G_\ga(\bbA)/G(\bbO )\big),\\}; 
\path[->](a-1-1) edge node[above] {$\simeq$} (a-1-2);
\path[->](a-2-1) edge node[above] {$\simeq$} (a-2-2);
\path[->](a-1-1) edge node[right] {} (a-2-1);
\path[->](a-1-2) edge node[right] {} (a-2-2);
\end{tikzpicture}
\] 
identifying the vertical maps as the level maps.
\xprop

Now assume that $k$ is a finite field. As a consequence of Proposition \ref{torsorprop} one naturally obtains integral models for moduli stacks
of $G$-shtukas on $X$ with level structures over $N$ via an isomorphism $\Sht_{\calG,I_\bullet}\;\overset{\cong}{\longto}\;\Sht_{(G,H,N),I_\bullet}$ (cf. \cite[§4.2.2]{MRR20} for precise definitions and details).

\section{\label{Topol}The \textquotedblleft topology\textquotedblright
of dilatations of affine schemes}

\subsection{\label{subsec:classical-top}Constructing smooth complex affine varieties
with controlled topology}

Dilatations have played an important role in complex affine algebraic
geometry during the nineties in connection to the construction and
study of exotic complex affine spaces \cite{KZ99,Zai00}, that is,
smooth algebraic $\mathbb{C}$-varieties $X$ of dimension $n$ whose
analytifications $X^{\mathrm{an}}$ are homeomorphic to the Euclidean
space $\mathbb{R}^{2n}$ endowed with its standard structure of topological
manifold but which are not isomorphic to the affine space $\mathbb{A}_{\mathbb{C}}^{n}$
as $\mathbb{C}$-varieties. 

In this context, dilatations appeared under the name \emph{affine
modifications} and were used as a powerful tool to produce from a
given smooth complex affine variety $X$ a new smooth complex affine
variety $X'=\mathrm{Bl}_{Z}^{D}X$ for which the homology or homotopy
type of the underlying topological manifold of the analytification
of $X'$ can be determined under suitable hypotheses in terms of those
of $X$ and of the center $\{[Z,D]\}$ of the dilatation. 

\vspace{0.1cm}

The study of the strong topology of affine modifications was initiated
in this context mainly by Kaliman through an analytic counterpart
of the notion of dilatation: 

\begin{dfn}[\cite{Ka94}]
Given a triple $(M,H,C)$ consisting of a complex analytic manifold
$M$, a closed submanifold $C$ of $M$ of codimension at least $2$
and a complex analytic hypersurface $H$ of $M$ containing $C$ in
its smooth locus, the \emph{Kaliman modification of $M$ along $H$ with
center at $C$} is the complex analytic manifold defined as the complement
$M'$ of the proper transform $H'$ of $H$ in the blow-up $\sigma_{C}:\hat{M}\to M$
of $M$ with center at $C$. 
\end{dfn}

In the case where $(M,H,C)$ is the analytification of a triple $(X,D,Z)$
consisting of a smooth algebraic $\mathbb{C}$-variety $X$, a smooth
algebraic sub-variety $Z$ of $X$ of codimenion at least two and
of a reduced effective Cartier divisor $D$ on $X$ containing $Z$
in its regular locus, the Kaliman modification of $(M,H,C)$ coincides
with the analytification of the dilatation $X'=\mathrm{Bl}_{Z}^{D}X$
of $X$ with center $\{[Z,D]\}$ of Section \ref{333}. 

\vspace{0.1cm}

Kaliman and Zaidenberg \cite{KZ99} developed a series of tools to
describe the topology of the analytifications of affine modifications
of smooth affine $\mathbb{C}$-varieties along principal divisors
$D$ with non-necessarily smooth centers. One of these provides in
particular a control on the preservation of the topology of the analytification
under affine modifications:
\begin{thm}[{\cite[Proposition 3.1 and Theorem 3.1]{KZ99}}]
\label{thm:Kaliman-Zaidenberg-MainThm} Let $X$ be a smooth connected affine
$\mathbb{C}$-variety and let $\{[Z,D]\}$ be a center on $X$ consisting
of a closed sub-scheme $Z$ of codimension at least $2$ and of a principal effective
divisor $D$ containing $Z$ as a closed subscheme. Let $\sigma:\tilde{X}=\mathrm{Bl}_{Z}^{D}(X)\to X$
be the dilatation of $X$ with center $\{[Z,D]\}$ and let $E$ be
its exceptional divisor. 

\noindent Assume that the following conditions are satisfied:

(i) The $\mathbb{C}$-variety $\tilde{X}=\mathrm{Bl}_{Z}^{D}(X)$
is smooth;

(ii) The divisors $E$ and $D$ are irreducible, $E=\sigma^{*}D$,
and the analytifications of $E_{\mathrm{red}}$ and $D_{\mathrm{red}}$
are topological manifolds.

\noindent Then the following properties hold:

(a) The homomorphism $\sigma_{*}^{\mathrm{an}}:\pi_{1}(\tilde{X}^{\mathrm{an}})\to\pi_{1}(X^{\mathrm{an}})$
induced by $\sigma^{\mathrm{an}}$ is an isomorphism;

(b) The homomorphism $\sigma_{*}^{\mathrm{an}}:H_{*}(\tilde{X}^{\mathrm{an}};\mathbb{Z})\to H_{*}(X^{\mathrm{an}};\mathbb{Z})$
induced by $\sigma^{\mathrm{an}}$ is an isomorphism if and only if
the homomorphism $\sigma|_{E,*}^{\mathrm{an}}:H_{*}(E_{\mathrm{red}}^{\mathrm{an}};\mathbb{Z})\to H_{*}(D_{\mathrm{red}}^{\mathrm{an}},\mathbb{Z})$
is. 
\end{thm}

Combining the above theorem with basic algebraic topology results (the fact that smooth complex affine algebraic varieties are homotopy equivalent to CW complexes \cite{AF59} and the Hurewicz and Whitehead theorems, see e.g. Theorem 4.32 and Theorem 4.5 in \cite{HaTop}),  the following can be obtained. 

\begin{cor}
\label{cor:Main-cor-Top} Under the assumptions of Theorem \ref{thm:Kaliman-Zaidenberg-MainThm},
assume in addition that $X^{\mathrm{an}}$ is a contractible smooth manifold,
that $Z_{\mathrm{red}}^{\mathrm{an}}$ is a topological sub-manifold of $D_{\mathrm{red}}^{\mathrm{an}}$ and
that the homomorphism $j_{*}^{\mathrm{an}}:H_{*}(Z_{\mathrm{red}}^{\mathrm{an}};\mathbb{Z})\to H_{*}(D_{\mathrm{red}}^{\mathrm{an}};\mathbb{Z})$
induced by the closed immersion $j:Z\hookrightarrow D$ is an isomorphisms. Then the analytification of $\tilde{X}=\mathrm{Bl}_{Z}^{D}(X)$ is
a contractible smooth manifold.
\end{cor}
\begin{proof} Once in possession of the aforementioned results of algebraic topology, the assertion readilly follows from the fact that the induced continuous map $E^{\mathrm{an}}_{\mathrm{red}}\to Z^{\mathrm{an}}_{\mathrm{red}}$ is a topological vector bundle.  If $Z$ were to be smooth or more generally a local complete intersection in $D$, this is a consequence of Proposition \ref{exceptional divisor}(2) which asserts the stronger property that $E\to Z$ is an algebraic vector bundle, but this remains true under the weaker assumptions made in the statement. 
\end{proof}

Having the flexibility to use as centers or divisors of modifications
schemes which are either non-reduced or whose analytifications are
not necessarily smooth manifolds but only topological manifolds is
particularly relevant for applications to the construction smooth
$\mathbb{C}$-varieties with contractible analytifications, as illustrated
by the following examples.
\begin{ex}[The tom Dieck - Petrie surfaces]
\label{exa:.-tom-Dieck-Petrie-surfaces} Let $p,q\geq2$ be a pair
of relatively prime integers and let $C_{p,q}\subset\mathbb{A}_{\mathbb{C}}^{2}=\mathrm{Spec}(\mathbb{C}[x,y])$
be an irreducible rational cuspidal curve with equation $x^{p}-y^{q}=0$.
The underlying topological space of $C_{p,q}^{\mathrm{an}}$ is a
contractible real topological surface, and hence, Corollary \ref{cor:Main-cor-Top}
applies to conclude that the analytification of the dilatation $S_{p,q}=\mathrm{Bl}_{(1,1)}^{C_{p,q}}\mathbb{A}_{\mathbb{C}}^{2}$
of $\mathbb{A}_{\mathbb{C}}^{2}$ along the principal Cartier divisor
$D=C_{p,q}$ with center at the closed point $Z=(1,1)\in C_{p,q}$
is a smooth contractible real $4$-manifold. The smooth affine surface
$S_{p,q}$, which can be described explicitly as the hypersurface
in $\mathbb{A}_{\mathbb{C}}^{3}=\mathrm{Spec}(\mathbb{C}[x,y,z])$
with equation 
\[
\frac{(xz+1)^{p}-(yz+1)^{q}}{z}=1,
\]
is not isomorphic to $\mathbb{A}_{\mathbb{C}}^{2}$ since, for instance,
it has non negative logarithmic Kodaira dimension \cite[Example 2.4]{Zai00}.
Moreover, the underlying real $4$-manifold of $S_{p,q}^{\mathrm{an}}$
is an example of a contractible $4$-manifold with non-trivial fundamental
group at infinity, hence non-homeomorphic to the standard euclidean
space $\mathbb{R}^{4}$. 
\end{ex}

\begin{ex}[Some Koras-Russell threefolds]
\label{exa:-KR-threefolds} Let again $p,q\geq2$ be a pair of relative
prime integers and consider for every $n\geq2$ the smooth hypersurface
$X_{p,q,n}$ in $\mathbb{A}_{\mathbb{C}}^{4}=\mathrm{Spec}(\mathbb{C}[x,y,z,w])$
with equation 
\[
x^{n}y+z^{p}+w^{q}+x=0.
\]
The restriction $\sigma_{p,q,n}:X_{p,q,n}\to\mathbb{A}_{\mathbb{C}}^{3}$
of the projection to the coordinates $x$, $z$ and $w$ expresses
$X_{p,q,n}$ as the dilatation of $\mathbb{A}_{\mathbb{C}}^{3}$ along the
principal divisor $D_{n}=\mathrm{div}x^{n}$ and with center at the
non-reduced closed sub-scheme $Z=Z_{p,q,n}$ of codimension $2$ with defining
ideal $$I_{p,q,n}=(z^{p}+w^{q}+x,x^{n})\subset\mathbb{C}[x,z,w].$$
The analytification of $Z_{\mathrm{red}}$ is a topological manifold
homeomorphic to the underlying topological space of the curve $C_{p,q}^{\mathrm{an}}$
of the previous example. Thus, Corollary \ref{cor:Main-cor-Top} applies
to conclude that $X_{p,q,n}^{\mathrm{an}}$ is a contractible real
$6$-manifold, hence, by a result of Dimca--Ramanujam, is diffeomorphic
to the standard euclidean space $\mathbb{R}^{6}$, see \cite[Theorem 3.2]{Zai00}. 

The interest in these affine threefolds $X_{p,q,n}$ was motivated
in the nineties by their appearance in the course of the study of
the linearization problem for actions of the multiplicative group $\mathbb{G}_{m,\mathbb{C}}$ on $\mathbb{A}_{\mathbb{C}}^{3}$
by Koras and Russell \cite{KR97}. One crucial question at that time
was to decide whether these threefolds were isomorphic to $\mathbb{A}_{\mathbb{C}}^{3}$
or not. The fact that none of them is isomorphic to $\mathbb{A}_{\mathbb{C}}^{3}$
was finally established by Makar-Limanov \cite{ML96} by commutative
algebra techniques. An interesting by-product of his proof is that
the dilatation morphism $\sigma_{p,q,n}:X_{p,q,n}\to\mathbb{A}_{\mathbb{C}}^{3}$
is equivariant with respect to the natural action of the group of
$\mathbb{C}$-automorphism of $X_{p,q,n}$, more precisely, $\sigma_{p,q,n}$
induces an isomorphism 
\[
\sigma_{p,q,n}^{*}:\mathrm{Aut}_{\mathbb{C}}(\mathbb{A}_{\mathbb{C}}^{3},\{[Z_{p,q,n},D_{n}]\})\to\mathrm{Aut}_{\mathbb{C}}(X_{p,q,n})
\]
between the subgroup $\mathrm{Aut}_{\mathbb{C}}(\mathbb{A}_{\mathbb{C}}^{3},\{[Z_{p,q,n},D_{n}]\})$
of $\mathrm{Aut}_{\mathbb{C}}(\mathbb{A}_{\mathbb{C}}^{3})$ consisting
of $\mathbb{C}$-automorphisms preserving the divisor and the center
of the dilatation $\sigma_{p,q,n}$ and the group $\mathrm{Aut}_{\mathbb{C}}(X_{p,q,n})$,
see \cite{MJ11}. 
\end{ex}

\subsection{Deformation to the normal cone}
\label{subsec:DefSpace}
A very natural class of dilatations which plays a fundamental role
in intersection theory is given by the affine version of the deformation
space $D(X,Y)$ of a closed immersion $Y\hookrightarrow X$ of schemes
of finite type over a fixed base scheme $S$ to its normal cone, \cite{Ful98},
\cite[\S 10]{Ro96}. Indeed, $D(X,Y)$ is simply the dilatation of
$X\times_{S}\mathbb{A}_{S}^{1}$ with divisor $D=X\times_{S}\{0\}_{S}$,
where $\{0\}_{S}$ denotes the zero section, and center $Z=Y\times_{S}\{0\}_S$.
In the affine setting, say $X=\mathrm{Spec}(A)$ and $Y=\mathrm{Spec}(A/I)$
for some ideal $I\subset A$, $D(X,Y)$ is the spectrum of the sub-algebra
$A[t]$-algebra
\[
A[\frac{(I,t)}{t}]\cong\sum_{n}I^{n}t^{-n}\subset A[t,t^{-1}].
\]
The composition $f:D(X,Y)\to\mathbb{A}_{S}^{1}$ of the dilatation
morphism $\sigma:D(X,Y)\to X\times_{S}\mathbb{A}_{S}^{1}$ with the
projection $\mathrm{p}_{2}:X\times_{S}\mathbb{A}_{S}^{1}\to\mathbb{A}_{S}^{1}$
is a flat morphism restricting to the trivial bundle $X\times_{S}(\mathbb{A}_{S}^{1}\setminus\{0\}_{S})$
over $\mathbb{A}_{S}^{1}\setminus\{0\}_{S}$ and whose fiber over
$\{0\}_{S}$ equals the normal cone $N_{Y/X}=\mathrm{Spec}(\bigoplus_{n\geq0}I^{n}/I^{n+1})$
of the closed embedding $Y\hookrightarrow X$ (see Proposition \ref{exceptional divisor}).
For regular immersions $Y\hookrightarrow X$ between smooth schemes
of dimension $n$ and $m$ over a field $k$, the deformation space
$D(X,Y)$ \'etale locally looks like the deformation space 
\[
\mathrm{Spec}(k[x_{1},\ldots,x_{m}][t][u_{1},\ldots,u_{m-n}]/(tu_{i}-x_{i})_{i=1,\ldots,m-n})\cong\mathbb{A}_{k}^{m+1}
\]
of the immersion of $\mathbb{A}_{k}^{n}$ as the linear subspace $\{x_{1}=\ldots=x_{m-n}=0\}$
of $\mathbb{A}_{k}^{m}=\mathrm{Spec}(k[x_{1},\ldots,x_{m}])$.\\

Deformation spaces of closed immersions between smooth affine $\mathbb{C}$-varieties
are a rich source of smooth affine $\mathbb{C}$-varieties
whose analytifications are contractible smooth manifolds: 
\begin{ex}
Given a smooth affine $\mathbb{C}$-variety $X$ such that $X^{\mathrm{an}}$
is contractible and a smooth subvariety $Y\subset X$ such that the
induced inclusion $Y^{\mathrm{an}}\subset X^{\mathrm{an}}$ is a topological
homotopy equivalence, Theorem \ref{thm:Kaliman-Zaidenberg-MainThm}
implies that the analytification of the deformation space $D(X,Y)$
is a contractible smooth manifold. 

For instance, the deformation spaces $D(\mathbb{A}_{\mathbb{C}}^{3},S_{p,q})\subset\mathbb{A}_{\mathbb{C}}^{5}$
of the tom Dieck - Petrie surfaces $S_{p,q}$ of Example \ref{exa:.-tom-Dieck-Petrie-surfaces}
are smooth affine $\mathbb{C}$-varieties of dimension $4$ whose
analytifications are all diffeomorphic to $\mathbb{R}^{8}$. In the
same way, for every Koras-Russell threefold $X_{p,q,n}\subset\mathbb{A}_{\mathbb{C}}^{4}$
in Example \ref{exa:-KR-threefolds}, the deformation space 
\[
D(\mathbb{A}_{\mathbb{C}}^{4},X_{p,q,n})\cong\{tu=x^{n}y+z^{p}+w^{q}+x\}\subset\mathbb{A}_{\mathbb{C}}^{6}
\]
is a smooth affine $\mathbb{C}$-variety whose analytification is
diffeomorphic to $\mathbb{R}^{10}$. It is not known whether these
deformation spaces are algebraically isomorphic to affine spaces. 
\end{ex}

Another technique, first introduced in a local differential or complex analytic setting by tom Dieck \cite{tD92} under the name \emph{hyperbolic modification}, 
and developped further in the algebraic context by tom Dieck \cite{tD92} and Kaliman-Zaidenberg, see e.g.  \cite[$\S$ 4.3]{Zai00}, allows the construction of
many new classes of smooth complex algebraic varieties whose analytifications are contractible manifolds. A special type of such hyperbolic modifications
can be conveniently rephrased in terms of a two-step dilatation construction which underlines its relation to deformation spaces:

\begin{ex}[Hyperbolic modifications as dilatations] Let $X$ be a smooth complex algebraic variety, let $i:H\hookrightarrow X$ be a prime Cartier divisor on $X$ and let $Z$ be an integral closed sub-scheme of $X$ contained in $H$.

 Let $f:D(X,H)\to \mathbb{A}^1_{\mathbb{C}}$ be the deformation space of $H$ in $X$ and let $\hat{j}:H\times_\mathbb{C}\mathbb{A}^1_{\mathbb{C}} \hookrightarrow D(X,H)$ be the canonical closed immersion of schemes over $\mathbb{A}^1_{\mathbb{C}}$ whose restriction over $\mathbb{A}^1_{\mathbb{C}}\setminus\{0\}$ is the product $j\times\mathrm{id}:H\times_{\mathbb{C}} (\mathbb{A}^1_{\mathbb{C}}\setminus\{0\})\hookrightarrow  X\times_{\mathbb{C}} (\mathbb{A}^1_{\mathbb{C}}\setminus\{0\})\cong D(X,H)|_{(\mathbb{A}^1_{\mathbb{C}}\setminus\{0\}))}$ and whose restriction over $\{0\}$ is the inclusion of $H$ as the zero section of the normal bundle $N_{H/X} \cong D(X,H)|_{\{0\}}$ of $H$ in $X$. Viewing $Z\times_\mathbb{C} \{0\} \subset H\times_{\mathbb{C}} \mathbb{A}^1_{\mathbb{C}}$ and $H\times_{\mathbb{C}} \mathbb{A}^1_{\mathbb{C}}$ as closed sub-schemes of $D(X,H)$ via $\hat{j}$, the \emph{hyperbolic modification of $X$ with locus $(Z,H)$} is then defined as the dilatation $h:\hat{X}\to D(X,H)$ with divisor $H\times_{\mathbb{C}} \mathbb{A}^1_{\mathbb{C}}$ and center $Z\times_\mathbb{C} \{0\}$. 
 
 The composition $\hat{f}=f\circ h: \hat{X}\to \mathbb{A}^1_\mathbb{C}$ is a flat morphism whose restriction over  $\mathbb{A}^1_{\mathbb{C}}\setminus\{0\}$ is isomorphic to the product $\mathrm{Bl}_Z^{H} X \times_{\mathbb{C}} (\mathbb{A}^1_{\mathbb{C}}\setminus\{0\})$ and whose fiber over $\{0\}$ is isomorphic to $\mathrm{Bl}_Z^H(N_H/X)$, where $H\subset N_H/X$ is viewed as the zero section of $N_H/X$. In particular, in the case where $N_{H/X}$ is the trivial line bundle, $\hat{f}^{-1}(\{0\})$ is isomorphic to $D(H,Z)$ and the morphism $\hat{f}:\hat{X}\to \mathbb{A}^1_{\mathbb{C}}$ provides a flat deformation of $\mathrm{Bl}_Z^H(X)$ to $D(H,Z)$. 
 
 Now assume that $X^{\mathrm{an}}$, $H^{\mathrm{an}}$ and $Z^{\mathrm{an}}$ are contractible topological manifolds and that $Z$ is contained in the regular locus of $H$. Appealing twice to Corollary \ref{cor:Main-cor-Top} yields the conclusion that $D(X,H)$ and then $\hat{X}$ are smooth varieties whose analytifications are contractible manifolds. In the same way, $\mathrm{Bl}_Z^{H} X$ is a smooth variety whose analytification is contractible. The variety $\mathrm{Bl}_Z^H(N_H/X)$ is readily verified to be smooth if and only if $H$ is smooth, the underlying topological space of its analytification being in any case contractible. In sum, $\hat{f}^{\mathrm{an}}:\hat{X}^{\mathrm{an}}\to \mathbb{C}$ can be interpreted as providing a foliation with contractible leaves of the contractible manifold $\hat{X}^{\mathrm{an}}$.  
 
 In \cite{tD92, Zai00}, this construction is mainly considered for triples $(X,H,Z)=(\mathbb{A}^n_\mathbb{C},H,\{0\})$, $n\geq 2$, where $H$ is the zero locus of a polynomial $g$ on $\mathbb{A}^n_{\mathbb{C}}=\mathrm{Spec}(\mathbb{C}[x_1,\ldots,x_n])$ having the origin $\{0\}$ as a zero of multiplicity one. For such triples, the hyperbolic modification $\hat{X}$ is isomorphic to $\mathbb{A}^{n+1}_\mathbb{C}=\mathrm{Spec}(\mathbb{C}[y_1,\ldots, y_n,t])$ and the family $\hat{f}:\hat{X}\to \mathbb{A}^1_{\mathbb{C}}$ concides with that defined by the unique polynomial $\hat{f}\in \mathbb{C}[y_1,\ldots, y_n,t]$ such that $g(y_1t,\ldots, y_nt)=t\hat{f}(y_1,\ldots, y_n,t)$.  For instance, taking for $g\in \mathbb{C}[x,y]$ the polynomial $(x+1)^{p}-(y+1)^q$, where $p,q\geq2$ is a pair of relatively prime integers, yields a family $\hat{f}:\mathbb{A}^3_{\mathbb{C}}\to \mathbb{A}^1_{\mathbb{C}}$ whose non-zero fibers are isomorphic to the tom Dieck - Petrie surfaces $S_{p,q}$ of Example \ref{exa:.-tom-Dieck-Petrie-surfaces}.
\end{ex}

More general versions of Kaliman and Zaidenberg techniques allow to
fully describe the singular homology of the analytification of the
deformation space $D(\mathbb{A}_{\mathbb{C}}^{n},Z)$ of a smooth
hypersurface $Z=Z(p)=\mathrm{Spec}(\mathbb{C}[x_{1},\ldots,x_{n}]/(p))$
of $\mathbb{A}_{\mathbb{C}}^{n}$ in terms of that of the analytification
of $Z$, namely 
\begin{propo}[{\cite[Proposition 4.1]{KZ99}}]
\label{prop:Homology-type-Def-space}For a smooth hypersurface $Z\subset\mathbb{A}_{\mathbb{C}}^{n}$,
$D(\mathbb{A}_{\mathbb{C}}^{n},Z)^{\mathrm{an}}$ is simply connected
and the inclusion $N_{Z/\mathbb{A}_{\mathbb{C}}^{n}}\hookrightarrow D(\mathbb{A}_{\mathbb{C}}^{n},Z)$
induces an isomorphism of reduced homology groups $\tilde{H}_{*}(D(\mathbb{A}_{\mathbb{C}}^{n},Z)^{\mathrm{an}};\mathbb{Z})\cong\tilde{H}_{*-2}(Z^{\mathrm{an}};\mathbb{Z})$. 
In particular, $D(\mathbb{A}_{\mathbb{C}}^{n},Z)^{\mathrm{an}}$ has
the reduced homology type of the $S^{2}$-suspension of $Z^{\mathrm{an}}$. 
\end{propo}

\subsection{Contractible affine varieties in motivic $\mathbb{A}^{1}$-homotopy
theory}

The possibility to import Kaliman and Zaidenberg techniques in the
framewok of Morel-Voevodsky $\mathbb{A}^{1}$-homotopy theory of schemes
\cite{MV99} has focused quite a lot of attention recently, especially
in the direction of the construction of $\mathbb{A}^{1}$-contractible
smooth affine varieties, motivated in part by possible applications
to the Zariski Cancellation Problem, see \cite{AO21} and the reference
therein for a survey. 

Very informally, one views in this context smooth schemes over a fixed
base field $k$ as analogous to topological manifolds, with the affine
line $\mathbb{A}_{k}^{1}$ playing the role of the unit interval,
and consider the corresponding homotopy category. More rigorously,
the $\mathbb{A}^{1}$-homotopy category $\mathrm{H}_{\mathbb{A}^{1}}(k)$
of $k$-schemes is defined as the left Bousfield localization of the
injective Nisnevich-local model structure on the category of simplicial
presheaves of sets on the category $\mathrm{Sm}_{k}$ of smooth $k$-schemes,
with respect to the class of maps generated by projections from the
affine line $\mathcal{X}\times_{k}\mathbb{A}_{k}^{1}\to\mathcal{X}$.
Isomorphisms in the homotopy category $\mathrm{H}_{\mathbb{A}^{1}}(k)$
are called $\mathbb{A}^{1}$-\emph{weak equivalences}, and a smooth
$k$-scheme $X$ is called $\mathbb{A}^{1}$-\emph{contractible} if
the structure morphism $X\to\mathrm{Spec}(k)$ is an isomorphism in
$\mathrm{H}_{\mathbb{A}^{1}}(k)$. 

\vspace{0.1cm}

The affine space $\mathbb{A}_{k}^{n}$ is by definition $\mathbb{A}^{1}$-contractible.
Since the analytification of an $\mathbb{A}^{1}$-contractible smooth
$\mathbb{C}$-variety is a contractible smooth manifold, smooth algebraic
$\mathbb{C}$-varieties with contractible analytifications provided
conversely a first natural framework to seek for interesting $\mathbb{A}^{1}$-contractible
affine varieties non isomorphic to affine spaces. A first step in
this direction was accomplished by Hoyois, Krishna and {\O}stv{\ae}r
\cite[Theorem 4.2]{HKO16} who used the underlying geometry associated
to the dilatations morphisms $\sigma_{p,q,n}:X_{p,q,n}\to\mathbb{A}_{\mathbb{C}}^{3}$
to verify that the Koras-Russell threefolds of Example \ref{exa:-KR-threefolds}
were $\mathbb{A}^{1}$-contractible possibly up to a finite number
of $\mathbb{P}^{1}$-suspensions, in the sense that for some $n\geq0$,
the suspension $(X_{p,q,n},o)\wedge(\mathbb{P}^{1})^{\wedge n}$ is
an $\mathbb{A}^{1}$-contractible object in $\mathrm{H}_{\mathbb{A}^{1}}(\mathbb{C})$,
where here, $X_{p,q,n}\subset\mathbb{A}_{\mathbb{C}}^{4}$ is considered
as a pointed smooth $\mathbb{C}$-scheme with distinguished point
$o=(0,0,0,0)$. 

\vspace{0.1cm} 

These first constructions motivated a more systematic study to obtain
$\mathbb{A}^{1}$-homotopic analogues of Kaliman and Zaidenberg's
topological comparison results for affine modifications. The best
counterparts of Theorem \ref{thm:Kaliman-Zaidenberg-MainThm} and
Corollary \ref{cor:Main-cor-Top} available so far are the following:
\begin{thm}[{\cite[Theorem 2.17]{DP019}}]
\label{thm:A1-Homotopical-KZ}Let $(X,D,Z,p)$ be a quadruple in $\mathrm{Sm}_{k}$
where $D$ is a Cartier divisor on $X$, $Z\subset D$ is a closed
subscheme and $p\in Z(k)$ is a $k$-rational point. Let $\sigma:\tilde{X}=\mathrm{Bl_{D}^{Z}X\to X}$ be
the dilatation of $X$ along $D$ with center at $Z$ and let $q\in \tilde{X}(k)$ be a $k$-rational point mapping to $p$.
We then view $\sigma$ as a morphism of pointed $k$-schemes $(\tilde{X},q)\to (X,p)$. 
Assume that the following conditions are satisfied:

(i) The supports of $D$ and of the exceptional divisor $E$ of $\sigma$
are irreducible;

(ii) The closed immersion $(Z,p)\hookrightarrow (D,p)$ is a pointed $\mathbb{A}^{1}$-weak
equivalence. 

\noindent Then there is a naturally induced pointed $\mathbb{A}^{1}$-weak
equivalence $\Sigma_{s}\sigma:\Sigma_{s}\tilde{X}\to\Sigma_{s}X$
between the simplicial $1$-suspensions of the pointed schemes $(\tilde{X},q)$ and $(X,p)$ respectively. 

In particular, if $X$ is $\mathbb{A}^{1}$-contractible then $\Sigma_{s}\tilde{X}$
is an $\mathbb{A}^{1}$-contractible sheaf. Moreover, a stronger conclusion
holds in the reverse direction: if $\tilde{X}$ is an $\mathbb{A}^{1}$-contractible 
then $X$ is $\mathbb{A}^{1}$-contractible. 
\end{thm}

\begin{cor}
\label{cor:DefSpace-A1-Cor}Let $i:Y\hookrightarrow X$ be a closed
immersion between $\mathbb{A}^{1}$-contractible smooth $k$-schemes, which we view as morphism
of pointed schemes for a choice of $k$-rational point $p\in Y(k)$.
Then the simplicial $1$-suspension $\Sigma_{s}D(X,Y)$ of the deformation
space $D(Y,X)$ of $Y$ in $X$ (viewed as a $k$-scheme pointed by $p \times {0}\subset D(Y,X)$) is $\mathbb{A}^{1}$-contractible. 
\end{cor}

\begin{proof} Since $X$ is $\mathbb{A}^1$-contractible, so is $X\times \mathbb{A}^1$ by $\mathbb{A}^1$-homotopy invariance. Since $Y$ is also $\mathbb{A}^1$-contractible,  $i\times \mathrm{id}:(Y\times \{0\},p\times \{0\}) \hookrightarrow (X \times \{0\}, p\times\{0\})$ is a pointed  $\mathbb{A}^1$-weak equivalence. Viewing $X\times \mathbb{A}^1$ as a $k$-scheme pointed by $p\times \{0\}\subset Y\times\{0\}\subset X\times \{0\}$, Theorem \ref{thm:A1-Homotopical-KZ} then implies that the dilatation morphism $\sigma : D(Y,X)=\mathrm{Bl}_{Y\times \{0\}}^{X\times \{0\}}(X\times \mathbb{A}^1)\to X\times \mathbb{A}^1$ induces an $\mathbb{A}^1$-weak equivalence between $ \Sigma_{s}D(X,Y)$ and the $\mathbb{A}^1$-contractible presheaf $\Sigma_{s} (X\times \mathbb{A}^1)$ in $\mathrm{H}_{\mathbb{A}^{1}}(k)$.
\end{proof}
In contrast with the results of subsection \ref{subsec:classical-top}
which can be applied to possibly singular triples $(X,D,Z)$, Theorem
\ref{thm:A1-Homotopical-KZ} and its corollary fundamentally depend
on smoothness hypotheses. In particular, Theorem \ref{thm:A1-Homotopical-KZ}
is not applicable to tom Dieck -Petrie surfaces and Koras-Russell
threefolds over $\mathbb{C}$ and their natural generalization over
other fields. It was nevertheless verified in \cite{DF18} by different
geometric methods that over any base field $k$ of characteristic
zero, the Koras-Russell threefolds $X_{p,q,n}=\{x^{n}y+z^{p}+w^{q}+x=0\}$
are indeed all $\mathbb{A}^{1}$-contractible. These provide in turn
when combined with Theorem \ref{thm:A1-Homotopical-KZ} and Corollary
\ref{cor:DefSpace-A1-Cor} the building blocks for the construction
of many other new examples of smooth affine $k$-varieties whose simplicial
$1$-suspensions are $\mathbb{A}^{1}$-contractible, among which some
can be further verified by additional methods to be genuinely $\mathbb{A}^{1}$-contractible,
see \cite[Section 4]{DP019}.

\vspace{0.1cm}

A more detailed re-reading of the notion of deformation to the normal
cone of a closed immersion $Y\hookrightarrow X$ between smooth $k$-schemes
gives rise to a notion of ``parametrized'' deformation space over
a smooth base $k$-scheme $W$, which is defined as a dilatation of
the scheme $Y\times_{k}W$ with appropriate center, see \cite[Construction 2.1.2]{ADO21}.
This leads to the following counterpart and extension of Proposition
\ref{prop:Homology-type-Def-space} in the $\mathbb{A}^{1}$-homotopic
framework:
\begin{thm}[{\cite[Theorem 2]{ADO21}}]
\label{thm:-Deformation=00003Dsuspension} Let $X$ be a smooth $k$-scheme,
let $\pi:X\to\mathbb{A}_{k}^{n}$ be a smooth morphism with a section
$s$. Assume that $\pi|_{\pi^{-1}(\mathbb{A}_{k}^{n}\setminus\{0\})}:\pi^{-1}(\mathbb{A}_{k}^{n}\setminus\{0\})\to\pi^{-1}(\mathbb{A}_{k}^{n}\setminus\{0\})$
is an $\mathbb{A}^{1}$-weak equivalence. Then there exists an induced
pointed $\mathbb{A}^{1}$-weak equivalence 
\[
(X,s(0))\sim(\mathbb{P}^{1})^{\wedge n}\wedge(\pi^{-1}(0),s(0)).
\]
In particular, the deformation space $D(X,Y)$ of a closed immersion
$(Y,\star)\hookrightarrow(X,\star)$ between pointed smooth $k$-schemes
is $\mathbb{A}^{1}$-weakly equivalent to $\mathbb{P}^{1}\wedge(Y,\star)$. 
\end{thm}

\begin{ex}
Let $Q_{2n}\subset\mathbb{A}_{k}^{2n+1}=\mathrm{Spec}(k[u_{1},\ldots,u_{n},v_{1},\ldots,v_{n},z])$
be the smooth $2n$-dimensional split quadric with equation $\sum_{i=1}^{n}u_{i}v_{i}=z(z+1)$.
The projection $\pi=\mathrm{pr}_{u_{1},\ldots,u_{,n}}:Q_{2n}\to\mathbb{A}_{k}^{n+1}$
is a smooth morphism restricting to a Zariski locally trivial $\mathbb{A}^{n}$-bundle
over $\mathbb{A}_{k}^{n}\setminus\{0\},$ hence to an $\mathbb{A}^{1}$-weak
equivalence over $\mathbb{A}_{k}^{n}\setminus\{0\}$, and having the
morphism 
\[
s:\mathbb{A}_{k}^{n}\to Q_{2n},\,(u_{1},\ldots,u_{n})\mapsto(u_{1},\ldots,u_{n},0,\ldots,0,0)
\]
as a natural section. On the other hand, $\pi^{-1}(0)$ is $\mathbb{A}^{1}$-weakly
equivalent to the disjoint union of $s(0)$ and of the point $p=(0,\ldots,0\ldots,-1)$.
Theorem \ref{thm:-Deformation=00003Dsuspension} thus renders the
conclusion that $(Q_{2n},s(0))$ is $\mathbb{A}^{1}$-weakly equivalent
to $(\mathbb{P}^{1})^{\wedge n}\wedge (p\sqcup s(0))\sim(\mathbb{P}^{1})^{\wedge n}$.
In particular $Q_{2n}$ provides a smooth $k$-scheme model of the
motivic sphere $(\mathbb{P}^{1})^{\wedge n}=S^{n}\wedge\mathbb{G}_{m,k}^{\wedge n}$. 
\end{ex}

\bigbreak\bigbreak
\noindent \par

\noindent 
\noindent Institut de Math\'ematiques de Bourgogne, UMR 5584 CNRS, Universit\'e de Bourgogne, F-21000, Dijon ; 
Laboratoire de Math\'ematique et Applications, UMR 7348 CNRS, Universit\'e de Poitiers, F-86000, Poitiers. \par
\noindent E-mail address: \texttt{adrien.dubouloz@math.cnrs.fr}
\noindent  

\medskip

\noindent 
\noindent Einstein Institute of Mathematics,
Edmond J. Safra Campus,
The Hebrew University of Jerusalem,
Givat Ram. Jerusalem, 9190401, Israel \par
\noindent E-mail address: \texttt{arnaud.mayeux@mail.huji.ac.il}
\noindent

\medskip

\noindent 
\noindent Institut Montpelliérain Alexander Grothendieck, Université de Montpellier. Montpellier, France \par
\noindent E-mail address: \texttt{joao\_pedro.dos\_santos@yahoo.com}

\end{document}